\documentclass[reqno,11pt]{amsart}

\usepackage[letterpaper,margin=1in,footskip=0.25in]{geometry}
\usepackage[dvipsnames]{xcolor}
\usepackage{amsfonts,amscd,amsmath,amsthm,amssymb,tikz-cd,enumitem}
\definecolor{chianti}{rgb}{0.6,0,0}
\definecolor{meretale}{rgb}{0,0,.6}
\definecolor{leaf}{rgb}{0,.35,0}
\usepackage{graphicx}
\usepackage{mathptmx}
\usepackage{verbatim}   
  
\usepackage{stmaryrd}
\usepackage{fancyvrb,newverbs}
\usepackage[colorlinks=true, hyperindex, citecolor=meretale, urlcolor=leaf, linkcolor=chianti]{hyperref}
\usepackage[capitalize]{cleveref}
\usepackage{tikz,tikz-cd}
\usetikzlibrary{arrows,matrix}
\theoremstyle{definition}

\newtheorem{Theoremx}{Theorem}
 
\newtheorem{thm}{Theorem}[section]
\Crefname{thm}{Theorem}{Theorems}
\Crefname{Theoremx}{Theorem}{Theorems}
\Crefname{figure}{Figure}{Figures}
\Crefname{claim}{Claim}{Claims}
\newtheorem{xmp}[thm]{Example}
\newtheorem{dff}[thm]{Definition}
\newtheorem{lem}[thm]{Lemma}

\newtheorem{rmk}[thm]{Remark}
\newtheorem{prp}[thm]{Proposition}

\newtheorem{qst}[thm]{Question}

\definecolor{cverbbg}{gray}{0.93}

\newcommand{\uline}[1]{\underline{#1}}

\newcommand{\fbp}[1]{\left[#1 \right]}

\def\R{\mathbf R}

\def\Hom{\operatorname{Hom}}

\def\fa{\mathfrak{a}}
\def\fb{\mathfrak{b}}
\def\fm{\mathfrak{m}}
\def\fn{\mathfrak{n}}

\def\fp{\mathfrak{p}}
\def\fq{\mathfrak{q}}
\def\fQ{\mathfrak{Q}}

\newcommand{\hsl}{\operatorname{HSL}}

\newcommand{\rperf}{R_{\operatorname{perf}}}

\newcommand{\mfq}{\fq}

\newcommand{\mfp}{\mathfrak{p}}

\newcommand{\ts}{\textsuperscript}

\newcommand{\perf}{\mathrm{perf}}
\newcommand{\qis}{\simeq_{\mathrm{qis}}}

\newcommand{\im}{\operatorname{im}}

\newcommand{\Spec}{\operatorname{Spec}}

\renewcommand{\phi}{\varphi}

\newcommand{\id}{\operatorname{id}}

\def\R{\mathbf{R}}

\newcommand{\sbt}{\,\begin{picture}(-1,1)(-1,-2)\circle*{2}\end{picture}\ }

\def\xra{\xrightarrow}
\def\ux{\underline{x}}

\colorlet{DG}{green!50!black}
\colorlet{DO}{orange!50!black}
\colorlet{DR}{red!50!black}
\colorlet{DB}{blue!50!black}
\colorlet{DY}{yellow!50!black}

\newcommand{\qlf}{(\mfq^{\lim})^F}

\def\qs{\fq^{F\textrm{-lim}}}
\def\qlim{\fq^{\textrm{lim}}}

\def\qi{\fq_{\leq i}}
\def\qin{\fq_{\leq i+1}}

\def\zdF{0^F_{H^d_\fm(R)}}

\makeatletter
\renewcommand*{\eqref}[1]{
  \hyperref[{#1}]{\textup{\tagform@{\ref*{#1}}}}
}
\makeatother

\def\perf{\textrm{perf}}

\tikzset{
    labl/.style={anchor=south, rotate=270, inner sep=.5mm}
}

\usepackage[backend=biber, style=alphabetic, datamodel=mrnumber,sorting=anyt,maxalphanames=5,maxbibnames=99]{biblatex}
\usepackage{hyperref}
\usepackage{xurl}
\hypersetup{breaklinks=true}

\DeclareFieldFormat{mrnumber}{
  MR\addcolon\space
  \ifhyperref
    {\href{http://www.ams.org/mathscinet-getitem?mr=#1}{\nolinkurl{#1}}}
    {\nolinkurl{#1}}}

\renewbibmacro*{doi+eprint+url}{
  \iftoggle{bbx:doi}
    {\printfield{doi}}
    {}
  \newunit\newblock
  \printfield{mrnumber}
  \newunit\newblock
  \iftoggle{bbx:eprint}
    {\usebibmacro{eprint}}
    {}
  \newunit\newblock
  \iftoggle{bbx:url}
    {\usebibmacro{url+urldate}}
    {}}

\renewbibmacro{in:}{}

\title{A Buchsbaum theory for Frobenius closure}

\author[Goel]{Kriti Goel}
\thanks{Goel was supported by grant CEX2021-001142-S (Excelencia Severo Ochoa) funded by MICIU/AEI/10.13039/501100011033 (Ministerio de Ciencia, Innovaci\'on y Universidades, Spain).}
\address{Kriti Goel: Department of Mathematics, University of Missouri-Columbia, Columbia, MO 65211 USA}
\email{kritigoel.maths@gmail.com}
\urladdr{\url{https://sites.google.com/view/kritigoel/}}

\author[Maddox]{Kyle Maddox}
\address{Kyle Maddox: Department of Mathematical Sciences, University of Arkansas, 850 West Dickson Street, Fayetteville, Arkansas 72701, United States}
\email{kmaddox@uark.edu}
\urladdr{\url{https://sites.google.com/view/kylemaddox/}}

\author[Miller]{Lance Edward Miller}
\address{Lance Edward Miller: Department of Mathematical Sciences, University of Arkansas, 850 West Dickson
Street, Fayetteville, Arkansas 72701, United States}
\email{lem016@uark.edu}
\urladdr{\url{http://www.lemiller.net/}}

\author[Quy]{Pham Hung Quy}
\thanks{Quy was partially supported by a fund of Vietnam National Foundation for Science and Technology Development (NAFOSTED)}
\address{Pham Hung Quy: Department of Mathematics, FPT University, Hanoi, Vietnam}
\email{quyph@fe.edu.vn}

\author[Simpson]{Austyn Simpson}
\address{Austyn Simpson: Department of Mathematics, Bates College, 3 Andrews Rd, Lewiston, ME 04240 USA}
\email{asimpson2@bates.edu}
\urladdr{\url{https://austynsimpson.github.io/}}


\begin{document}

\begin{abstract}
We give a partial characterization for when the difference $e(\fq)-\ell_R(R/\fq^F)$ is independent of the choice of parameter ideal $\fq\subseteq R$ in an excellent equidimensional local ring $(R,\fm)$ of prime characteristic $p>0$. Here, $\fq^F$ is the Frobenius closure of $\fq$ and $e(\fq)$ denotes the Hilbert--Samuel multiplicity of $\fq$. In addition to ideal-theoretic equivalences, our characterization involves the derived category and is motivated by Schenzel's criterion of the Buchsbaum property as well as similar results of Ma--Quy in the setting of tight closure. 
\end{abstract}


\maketitle

\section{Introduction}

The Buchsbaum condition defines a singularity type for Noetherian local rings that has enjoyed a long history of broad influence in algebraic geometry and commutative algebra \cite{Got80, Sch82, CTS78, SV78}; for a more thorough introduction, see the monograph \cite{SV86}. A local ring $(R,\fm)$ is \emph{Buchsbaum} if the quantity
\begin{align}
\ell_R(R/\fq) - e(\fq)\label{intro-1}
\end{align} (where $\fq\subseteq R$ is an ideal generated by a system of parameters, and $e(\fq)$ is its multiplicity) does not depend on the choice of parameter ideal. This is a minimal generalization of the Cohen--Macaulay property, which is equivalent to the equality $e(\fq) = \ell_R(R/\fq)$ for all parameter ideals $\fq$. These rings have isolated non-Cohen--Macaulay loci, and putting uniform constraints on \eqref{intro-1} also imposes strong restrictions on the structure of their Hilbert functions \cite{GN03,GO10}. The Buchsbaum condition can be reformulated in a number of ways, and each perspective has shaped subsequent investigations. 
\begin{thm}\label{intro:thm-buchsbaum}
    Let $(R,\fm,k)$ be a $d$-dimensional Noetherian local ring and let $\fq\subseteq R$ be an ideal generated by a system of parameters. The following conditions are equivalent.
\begin{enumerate}[label=(\alph*)]
    \item $R$ is Buchsbaum.\label{Schenzel:a}
    \item The quantity $\ell_R(\fq^{\lim}/\fq)$ is independent of $\fq$.\label{Schenzel:b}
    \item For every system of parameters $x_1,\ldots, x_d\in R$, the containment $\fm ((x_1,\ldots, x_i):_R x_{i+1})\subseteq (x_1,\ldots, x_i)$ holds for every $i<d$. \label{Schenzel:c}
    \item The truncation $\tau^{<d}\mathbf{R}\Gamma_\fm R$ is quasi-isomorphic to a complex of $k$-vector spaces.\label{Schenzel:d}
\end{enumerate}
\end{thm}
The final condition \ref{Schenzel:d} is due to Schenzel \cite{Sch82}, and through this lens it is clear that the lower local cohomology modules $H^{i<d}_\fm(R)$ of $R$ are finite dimensional $k$-vector spaces (see \cref{sec:truncations} for details about the truncation operations).

In prime characteristic, the Buchsbaum property displays striking relationships with $F$-singularity theory. For example, if $(R,\fm)$ has an isolated non-Cohen--Macaulay locus and is $F$-injective, then $R$ is Buchsbaum \cite{Ma15}. If we additionally assume that the non-$F$-rational locus of $R$ is isolated, then the constancy of \eqref{intro-1} can be upgraded to constancy of
\begin{align}
    e(\fq) - \ell_R(R/\fq^*)\label{intro-2}
\end{align}
where $\fq^*$ is the \emph{tight closure} of $\fq$ \cite{BMS18,Quy18}. Rings for which \eqref{intro-2} is a constant have since been referred to as \emph{tight Buchsbaum}, and it is proven in \cite{MQ22} that this notion generalizes $F$-rationality in many analogous ways to how the Buchsbaum property generalizes Cohen--Macaulayness as in \cref{intro:thm-buchsbaum}. For instance, similar to how the equality $e(\fq) = \ell_R(R/\fq)$ detects Cohen--Macaulayness, $F$-rationality may be detected by the equality $e(\fq) = \ell_R(R/\fq^*)$ (i.e. when the quantity in \eqref{intro-2} is zero; see \cite{GN01}). Note that excellent $F$-rational local rings are always Cohen--Macaulay. In summary, Ma and Quy demonstrate in \cite{MQ22} a picture for tight Buchsbaum rings which closely mirrors that of \cref{intro:thm-buchsbaum}.
\begin{thm}[{\cite[Theorem 5.1]{MQ22}}]\label{thm:ma-pham}
For an unmixed, $d$-dimensional excellent local ring  $(R,\fm,k)$ of prime characteristic $p>0$, the following are equivalent:
\begin{enumerate}[label=(\alph*)]
    \item $e(\fq)-\ell_R(R/\fq^*)$ is constant for all parameter ideals $\fq$.
    \item $\ell_R(\fq^*/\fq)$ is constant for all parameter ideals $\fq$.
    \item For every parameter ideal $\fq$, $\fm\fq^*\subseteq \fq$; that is, $\fm \subseteq \tau_{\text{par}}(R),$ the parameter test ideal.
    \item The tight closure truncation $\tau^{<d,*}\mathbf{R}\Gamma_\fm(R)$ is quasi-isomorphic to a complex of $k$-vector spaces.
\end{enumerate}  
\end{thm}
Recently, the tight Buchsbaum condition has been shown to restrict the \emph{tight} Hilbert coefficients, further deepening the analogy to the classical story \cite{DQV23,HLQ25} (see also \cite{MQ24}).

Our perspective in this article is to more deeply focus on how Buchsbaum conditions interact with $F$-injectivity. In a Cohen--Macaulay ring, $F$-injectivity is characterized by the equality $\fq=\fq^F$ for all parameter ideals $\fq$ \cite[Corollary 3.9]{QS17}, where $\fq^F$ denotes the \emph{Frobenius closure} of $\fq$; in analogy with \eqref{intro-1} and \eqref{intro-2}, we investigate in this article when the quantity
\begin{equation}
    e(\fq) - \ell_R(R/\fq^F)\label{intro:3}
\end{equation}
is independent of $\fq$. Our goal will be to relate this independence to a similar finiteness condition on the truncation ``up to Frobenius closure," denoted $\tau^{<d,F}\R\Gamma_\fm(R)$. However, despite the deep relationship between tight and Frobenius closure, the new picture turns out to be much more subtle for $\fq^F$ than for $\fq^*$. One complication is that the limit closure arising as $\fq^{\lim}/\fq =\ker(R/\fq \to H_\fm^d(R))$
(see also \eqref{def:lim-closure}) and appearing in \cref{intro:thm-buchsbaum} is not always comparable with the Frobenius closure $\fq^F$, in contrast to the containment $\fq^{\lim}\subseteq\fq^*$ which holds for all excellent equidimensional local rings \cite[Theorem 2.3]{Hun98}. In addition to this defect, the Frobenius closure generally does not enjoy \emph{colon capturing} in contrast to  tight closure. Hence, the naive adaptation to Frobenius closure in \eqref{intro:3} does not yield a theory closely related to Buchsbaum singularities. To handle these obstacles, we consider the ideal $\qs := \fq^F + \fq^{\lim}$ and the closely related $\qlf$ which are both better suited for analyzing Buchsbaum singularities (and that latter of which \emph{does} satisfy colon capturing; see \cref{thm: colon capturing for qlf}). A simplified version of our main theorem is summarized as follows. \vspace{-.1in}

\textbf{\begin{Theoremx}[\cref{thm:(a)<=>(b),thm:mainthm-restated}]\label{thm:mainthm}
For an excellent $d$-dimensional unmixed local ring $(R,\fm,k)$ of prime characteristic $p>0$, consider the following conditions:
\begin{enumerate}[label=(\alph*)]
\item\label{intro:a} The quantity $e(\mfq)-\ell_R(R/\qs)$ is a constant independent of parameter ideal $\fq\subseteq R$.
\item\label{intro:b}The quantity $\ell_R(\qs/\fq)$ is independent of parameter ideal $\fq\subseteq R$.
\item\label{intro:c} $\fm (\qs) \subseteq \fq$ for all $\fq$ generated by system of parameters. 
\item\label{intro:d} the $F$-truncation $\tau^{<d,F}\mathbf{R} \Gamma_\fm(R)$ is quasi-isomorphic to a complex of $k$-vector spaces.
\end{enumerate}
Then \ref{intro:a} and \ref{intro:b} are equivalent. Moreover, if $R$ is $F$-finite and weakly $F$-nilpotent, then all four conditions are equivalent.
\end{Theoremx}}
We call local rings satisfying the equivalent conditions \ref{intro:a} or \ref{intro:b} of \cref{thm:mainthm} \emph{$F$-Buchsbaum}. We prove that an $F$-Buchsbaum ring is Buchsbaum (\Cref{thm:(a)<=>(b)}), and further satisfies constancy $e(\fq) - \ell_R(R/\fq^F)$ (\Cref{lem: qs/q constant length implies qF/q constant length}). Furthermore, when $R$ is Buchsbaum and weakly $F$-nilpotent, the constancy of $e(\fq) - \ell_R(R/\fq^F)$ implies that $R$ is $F$-Buchsbaum by \Cref{prop:wfn-q-lim-F}.

\begin{figure}[htp!]
    \centering
    \begin{equation*}
    \begin{tikzcd}[column sep=huge]
        \text{$F$-rational}\arrow[r,Rightarrow]\arrow[d,Rightarrow] & \text{tight Buchsbaum} \arrow[r,Rightarrow,"\text{\cite[Theorem 1.1]{GN02}}"] \arrow[d,swap,Rightarrow,"\text{\Cref{prp: tbb implies fbb}}" {xshift=-2px}] & \text{$F$-rational on }\Spec^\circ(R)\arrow[d,Rightarrow]\arrow[l,bend right,sloped,Rightarrow,"+F\text{-injective}"]\arrow[l,bend right,sloped,Rightarrow,"\text{\cite{BMS18, Quy18}}" {yshift=-17px}]\\
        \text{CMFI}\arrow[r,Rightarrow]\arrow[d,Rightarrow] & \text{$F$-Buchsbaum}\arrow[r,bend left=10, Rightarrow,"\text{\cref{thm: CMFI-pun}}" {yshift=2px}]\arrow[d,Rightarrow,swap,"\text{\Cref{thm:(a)<=>(b)}}" {xshift=-2px}]& \text{CMFI on }\Spec^\circ(R)\arrow[d,Rightarrow] \arrow[l,bend left=10,sloped,Rightarrow,"+F\text{-injective}"]\arrow[l,bend left=10,sloped,Rightarrow,"\text{\cref{rmk: f-inj  + buchs = f-buchs}}" {yshift=-12px}] \\
        \text{Cohen--Macaulay}\arrow[r,Rightarrow]&\text{Buchsbaum}\arrow[r,Rightarrow]&\text{Cohen--Macaulay on } \Spec^\circ(R) \arrow[l,bend left,sloped,Rightarrow,"+F\text{-injective}"',"\text{\cite{Ma15}}"]
    \end{tikzcd}
\end{equation*}
\caption{Implications between Buchsbaum-type conditions for excellent, unmixed local rings of prime characteristic}
    \label{diagram:implications}
\end{figure}
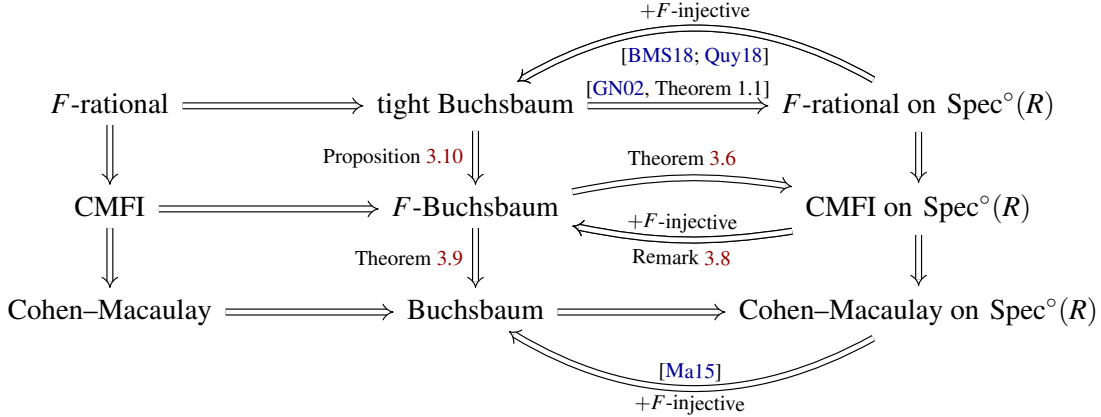

The relationships between the notions described in the introduction are summarized in \cref{diagram:implications}. We also find a simplification (\cref{cor:gor}) which aids in checking whether a Cohen--Macaulay ring is $F$-Buchsbaum or tight Buchsbaum. We apply this criterion to several examples which demonstrate that the implications in \cref{diagram:implications} are generally not reversible.

Finally, in the course of proving \cref{thm:mainthm}, we find an analogue of a theorem by Goto--Nakamura in terms of $\qs$ for local rings with isolated non-CMFI loci.

\begin{Theoremx}[\Cref{thm: CMFI-pun}]\label{thm:mainthm-B}
    Let $(R,\fm)$ be a $d$-dimensional excellent equidimensional local ring of prime characteristic $p>0$. The following are equivalent.
    \begin{enumerate}[label=(\alph*)]
        \item For all $\fp\in\Spec(R)\smallsetminus\{\fm\}$, the localization $R_\fp$ is Cohen--Macaulay and $F$-injective.
        \item $R$ is generalized Cohen--Macaulay and $\ell_R\left( 0^F_{H^d_\fm(R)}\right)<\infty$.
        \item $\sup\{\ell_R( \qs /\fq)\mid \fq \text{ is a parameter ideal}\}<\infty$.
    \end{enumerate}
\end{Theoremx}

The article is organized as follows. \cref{sec:preliminaries} describes the truncation operations, as well as their counterparts for tight and Frobenius closures. We prove Theorems \ref{thm:mainthm-B} and \ref{thm:mainthm} in \cref{sec:CMFI,sec:main} respectively, and we conclude with several examples in \cref{sec:examples}.

\subsection*{Acknowledgments} We are grateful to Alessandra Costantini and Neil Epstein for early contributions to this article. We also thank Linquan Ma for several insightful conversations and Kazuma Shimomoto for comments on an earlier draft of this article which led to significant improvements.

\section{Preliminaries}\label{sec:preliminaries}

By a triple $(R,\fm,k)$ we mean a Noetherian local ring $R$ with unique maximal ideal $\fm$ and residue field $k = R/\fm$. In such a situation, set $d:=\dim(R)$. Most situations we consider in this article assume that $R$ has prime characteristic $p>0$. For any fixed system of parameters $x_1,\ldots,x_d$ of $R$ with $\fq = (x_1,\ldots,x_d)$,  we denote $\fq_n = (x_1^n,\ldots,x_d^n)$ and $\qi = (x_1,\ldots,x_i)$ for $i < d$.  We denote by $R^\circ := R \smallsetminus \bigcup_{\mfp \in \min(R)} \mfp$ the complement of the union of the minimal primes of $R$, and we denote by $\Spec^\circ(R):=\Spec(R)\smallsetminus\{\fm\}$ the \emph{punctured spectrum} of $R$. We denote the length of an $R$-module by $\ell_R(-)$.

\subsection{Closure operations in prime characteristic}\label{subsec: closures}

Throughout this subsection, $(R,\fm)$ will be a local ring of prime characteristic $p>0$. We will write $F^e:R \rightarrow R$ for the $e$\ts{th}-iterate of the Frobenius endomorphism and $F_*^e(R)$ for the $R$-module obtained by restricting scalars along $F^e$.

We now recall the notions of tight and Frobenius closure, following \cite[Section 8]{HH90}. Let $M$ be an $R$-module and let $N\subseteq M$ be a submodule. The \emph{Frobenius closure of $N$ in $M$}, denoted $N^F_M,$ is the $R$-submodule of elements $\eta\in M$ mapping to zero under
\begin{align*}
    M \rightarrow M/N \stackrel{\id \otimes F^e}{\rightarrow} M/N\otimes_R F^e_*(R)
\end{align*} for some $e \in \mathbb{N}$. 
The \emph{tight closure of $N$ in $M$}, denoted by $N^*_M,$ is the $R$-submodule of elements in the kernel of the composition 
\begin{align*}
M \rightarrow M/N \rightarrow M/N\otimes_R F^e_*(R) \stackrel{\id \otimes \cdot F^e c}{\rightarrow} M/N\otimes_R F^e_*(R)
\end{align*}
    for some $c \in R^\circ$ and for all $e \gg 0$. We will primarily study these notions in the case of the zero submodule of the local cohomology modules $H^i_\fm(R)$, i.e., the tight closure $0^*_{H^i_\fm(R)}$ and Frobenius closure $0^F_{H^i_\fm(R)}$ of zero. For the case of ideals $I\subseteq R$ on the other hand, the tight and Frobenius closures are more explicit: 
\begin{align*}
I^F &= \{ x \in R \mid x^{p^e} \in I^{\fbp{p^e}}\text{ for all }e\gg 0\}, \\
I^* &= \{ x \in R \mid \text{there is a } c \in R^\circ \text{ such that } cx^{p^e} \in I^{\fbp{p^e}} \text{ for all } e \gg 0\} 
\end{align*}
where $I^{[p^e]}$ denotes the ideal $F^e(I)R$.

We will also need to work with a relative notion of these closure operations which we now review (cf. \cite[Section ~4]{PQ19}). If $I\subseteq R$ is an ideal, we can factor the map $F^e:R/I\rightarrow R/I$ as
\begin{equation}
        \begin{tikzcd}
            R/I \arrow{rr}{F^e}\arrow[swap]{dr}{\rho_e} & & R/I \\
            \, & R/I^{\fbp{p^e}} \arrow[swap]{ur}{\pi_e} &\,
        \end{tikzcd}\label{eq:relative-frob}
\end{equation}
where $\rho_e(x+I) = x^{p^e} +I^{\fbp{p^e}}$ and $\pi_e$ is the natural projection map. The $p^e$-linear map $\rho_e$ is called the \emph{relative Frobenius action} on $R/I$. These maps then induce maps (which we also call $\rho_e$) on the local cohomology modules
\begin{align*}
\rho_e: H^i_\fm(R/I)\rightarrow H^i_\fm(R/I^{\fbp{p^e}})
\end{align*}
and we call $$0^\rho_{H^i_\fm(R/I)}:=\bigcup_e \ker(\rho_e: H^i_\fm(R/I)\rightarrow H^i_\fm(R/I^{\fbp{p^e}}))$$ the \emph{relative Frobenius closure of zero in $H^i_\fm(R/I)$}. We also have the \emph{relative tight closure of zero in $H^i_\fm(R/I)$}, defined by 
\begin{align*}
    0^{*\rho}_{H^i_\fm(R/I)} := \left\lbrace\left. \xi \in H^i_\fm(R/I) \,\right|\, \exists c \in R^\circ \text{ such that } c\rho_e(\xi) = 0 \text{ }\forall e \gg 0\,\right\rbrace.
\end{align*}

Recall that $R$ is {\it weakly $F$-nilpotent} provided $H_\fm^i(R)$ is nilpotent, i.e., $0_{H_\fm^i(R)}^{F} = H_\fm^i(R)$ for all $i < \dim (R)$. A weakly $F$-nilpotent ring for which $0_{H_\fm^d(R)}^{F} = 0_{H_\fm^d(R)}^{*}$ is called {\it $F$-nilpotent}. We call $R/I$ {\it relatively weakly $F$-nilpotent} provided  $0_{H_\fm^i(R/I)}^{\rho} = H_\fm^i(R/I)$ for all $i < \dim (R/I)$. Weak $F$-nilpotence can also be characterized in terms of the Cohen--Macaulayness of $\rperf$, which we will make frequent use of in the course of proving the derived category implications in \cref{thm:mainthm}.

\begin{thm}[{\cite[Proposition 12.22]{MP25}}]\label{thm:wfn-rperf}
    A $d$-dimensional local ring $(R,\fm,k)$ of prime characteristic $p>0$ is weakly $F$-nilpotent if and only if $\rperf$ is a balanced big Cohen--Macaulay $R$-algebra. In particular, if $R$ is weakly $F$-nilpotent, then $H^i_\fm(\rperf)=0$ for all $i<d$.
\end{thm}

\subsection{Local cohomology and truncations}\label{sec:truncations} In this subsection, continue to let $(R,\fm,k)$ be a $d$-dimensional local ring of prime characteristic $p>0$ and let $D(R)$ denote the derived category of $R$-modules. We denote by $\qis$ quasi-isomorphisms in $D(R)$. For $C^{\sbt}$ a cochain complex representing a class in $D(R)$, denote by $h^i(C^{\sbt})$ its $i$\ts{th}-cohomology. Set $\R\Gamma_\fm(R)$ to be the complex in $D(R)$ with $h^i(\R\Gamma_\fm(R)) \cong H_{\fm}^i(R)$ for all $i$. Denote also, for each $i$, the natural truncations $\tau^{< i} \R\Gamma_{\fm}R$ or $\tau^{\leq i} \R\Gamma_{\fm}R$. By abuse of notation, we write $C^{\sbt} \in D(k)$ to mean that the complex $C^{\sbt}$ in $D(R)$ is quasi-isomorphic to a complex of $k$-vector spaces.

Next, we describe the truncation of $\R\Gamma_\fm(R)$ \emph{up to tight closure} as developed in \cite{BMS18,MQ22}, in addition to the analog we define for Frobenius closure that appears in \cref{thm:mainthm}.

\begin{dff}
Let $(R,\fm,k)$ be a $d$-dimensional reduced, equidimensional local ring of prime characteristic $p>0$. The \emph{tight closure truncation} of $\R\Gamma_\fm R$, denoted $\tau^{< d, *}\R\Gamma_\fm R$, is defined to be the object in $D(R)$ fitting into the exact triangle 
$$\tau^{< d, *}\R\Gamma_\fm R \to \R\Gamma_\fm R \to H^d_\fm(R)/0^*_{H^d_\fm(R)}[-d] \xrightarrow{+1}$$ so that $h^i(\tau^{< d,*}\R\Gamma_\fm R) \cong H_\fm^i(R)$ for $i < d$ and $h^d(\tau^{< d, *}\R\Gamma_\fm R) \cong 0^*_{H^d_\fm(R)}$. Similarly, we define the \emph{Frobenius closure truncation} of $\R\Gamma_\fm R$, denoted $\tau^{< d, F}\R\Gamma_\fm R$, to be the object in $D(R)$ fitting into the exact triangle
$$\tau^{< d, F}\R\Gamma_\fm R \to \R\Gamma_\fm R \to H^d_\fm(R)/0^F_{H^d_\fm(R)}[-d] \xrightarrow{+1}$$ so that $h^i(\tau^{< d,F}\R\Gamma_\fm R) \cong H_\fm^i(R)$ for $i < d$ and $h^d(\tau^{< d, F}\R\Gamma_\fm R) \cong 0^F_{H^d_\fm(R)}$. 
\end{dff}

Note for any parameter ideal $\fq = (x_1,\ldots,x_d)$ with $x=x_1\cdot x_2 \cdots x_d$, the kernel of the natural map $R/\fq \to H_\fm^d(R)$ is $\fq^{\lim}/\fq$, where the ideal
\begin{equation}
\fq^{\lim} := \bigcup_{n >0} (x_1^{n+1}, \cdots, x_d^{n+1}) :_R x^n\label{def:lim-closure}
\end{equation}
is called the \emph{limit closure of $\fq$}. Since this map depends only on $\fq$ and $R$, it is independent of the generating set for $\fq$. When $R$ is Buchsbaum, we can also calculate $\qlim$ as \[
\qlim = \mfq + \sum_{i=1}^d (x_1,\ldots,\widehat{x_i},\ldots,x_d):x_i 
\] by \cite[Theorem~4.7]{Got83}. When $R$ is excellent and equidimensional, colon-capturing for tight closure implies $\fq^{\lim} \subseteq \fq^*$.

\begin{rmk}\label{rmk: closure containments}
Suppose $(R,\fm)$ is excellent and equidimensional. Connections between parameter ideals and local cohomology allow us to interpret the Frobenius action on local cohomology through the lens of ideal closures of parameter ideals.
\begin{enumerate}[label=(\alph*)]
    \item \cite[Remark ~7.2(2)]{QS17} The inclusion $\fq \subseteq \qlim$ is equality for some (equivalently for all) parameter ideal(s) $\fq\subseteq R$ if and only if $R$ is Cohen--Macaulay. \label{remark:parameter ideals-2}
    \item \cite[Definition ~6.8]{QS17} If $\fq = \fq^F$ for all parameter ideals $\fq\subseteq R$ then $R$ is said to be \emph{parameter $F$-closed}. If so, $R$ is $F$-injective, see \cite[Theorem~A]{QS17}. \label{remark:parameter ideals-1}
\end{enumerate}
\end{rmk}

\subsection{Cohen--Macaulay and Buchsbaum type conditions} If $(R,\fm)$ is a local Buchsbaum ring, then the equivalent conditions of \cref{intro:thm-buchsbaum} imply that $\tau^{<d}\mathbf{R}\Gamma_\fm(R)$ has finite length cohomology; this is equivalent to $R_\fp$ being Cohen--Macaulay for all $\fp\in\Spec^\circ(R)$ provided that $R$ is excellent and equidimensional \cite[Satz 2.5 and Satz 3.8]{CTS78}. Such rings are said to be \emph{generalized Cohen--Macaulay}. Recall that a ring $R$ is \textit{equidimensional} if $\dim (R/\mfp) = \dim (R)$ for all minimal primes $\mfp$ of $R$. Furthermore, $R$ is said to be \emph{unmixed} if this equality holds for all associated primes of $R$.

Note that $R$ is Buchsbaum if and only if $\widehat{R}$ is Buchsbaum. There is a similar notion for modules, where we say that a finitely generated $R$-module $M$ is \emph{Buchsbaum} if the quantity $\ell_R(M/\fq M)-e(\fq,M)$ is constant as $\fq$ varies over all ideals generated by systems of parameters for $M$.

As discussed in the introduction, our aim is to give a Frobenius closure variant of tight Buchsbaum rings (that is, rings satisfying the conclusions of \cref{thm:ma-pham}). which we start in more detail in Section~\ref{sec:CMFI}. We expect that the next lemma is well known to experts, but we are unaware of a reference.

\begin{lem}\label{lem: buchs implies m q lim in q}
An equidimensional local ring $(R,\fm)$  of dimension $d>0$ is $R$ is Buchsbaum if and only if $\fm \fq^{\lim} \subseteq \fq$ for all parameter ideals $\fq \subseteq R$.
\end{lem}

\begin{proof} First, suppose $R$ is Buchsbaum. Let $x_1,\ldots,x_d$ be a system of parameters of $R$ and write $\fq = (x_1,\ldots,x_d)$ with $x=x_1\cdots x_d$. By \cite[Theorem~4.7]{Got87}, $\fm ((x_1^n,\ldots,x_d^n):x^{n-1})\subseteq (x_1^n,\ldots,x_d^n)$ for all $n \ge 1$. Since $\fq^{\lim}$ is the directed union of these colon ideals, we have $\fm \fq^{\lim} \subseteq \fq$.

For the converse, if $d=1$ and $x\in R$ is a parameter, then $(x)^{\lim} = 0:x^\infty+(x)$. Hence, for all $n$, we have $(x^n)^{\lim} = 0:x^{\infty} + (x^n)$. Then, $\fm(0:x) \subseteq \fm (x^n)^{\lim} \subseteq (x^n)$ for all $n$ by hypothesis. Therefore $\fm(0:x) \subseteq \cap_n (x^n) = (0)$, so $R$ is Buchsbaum. If $d\geq 2$ and $\fq=(x_1,\dots, x_d)$ is a parameter ideal, then $$((x_1,\ldots,x_{d-1}):x_d^\infty)+(x_d)\subseteq \fq^{\lim}$$ and hence
\begin{align*}
   (x_1,\ldots,x_{d-1}):x_d^\infty = &\bigcap\limits_{n>0} ((x_1,\ldots,x_{d-1}):x_d^\infty)+(x_d^n)\\
   \subseteq&\bigcap\limits_{n>0} (x_1,\ldots,x_{d-1},x_d^n)^{\lim}.
\end{align*}
Further
\begin{align*}
    \fm((x_1,\ldots,x_{d-1}):x_d^\infty)\subseteq & \bigcap\limits_{n>0} \fm(x_1,\ldots,x_{d-1},x_d^n)^{\lim}\\
    \subseteq & \bigcap\limits_{n>0} (x_1,\ldots,x_{d-1},x_d^n)\\
    =&(x_1,\ldots,x_{d-1})
\end{align*}
so $\fm((x_1,\ldots,x_{d-1}):x_d)\subseteq (x_1,\dots, x_{d-1})$ and $R$ is Buchsbaum (see \cite[Chapter I, Proposition 1.10]{SV86}).
\end{proof}

\begin{lem}\label{lem: buchs imples qlim/q iso to lc}
    Let $(R,\fm,k)$ be a Buchsbaum local ring of dimension $d>0$ and $\mfq$ a parameter ideal of $R$. We have the following isomorphism \[
    \dfrac{\qlim}{\mfq} \cong \bigoplus_{i=0}^{d-1} H^i_\fm(R)^{\binom{d}{i}},
    \] as $k$-vector spaces. If $R$ is additionally of prime characteristic $p>0$, then \[
    \dfrac{\qlim\cap \fq^F}{\mfq} \cong \bigoplus_{i=0}^{d-1} {0^F_{H^i_\fm(R)}}^{\binom{d}{i}}.
    \]
\end{lem}

\begin{proof}
    Fix a parameter ideal $\mfq=(x_1,\ldots,x_d)$ of $R$. For $1 \le t \le d$, write $R_t = R/(x_1,\ldots,x_t)$ and $\mfq_t = \mfq R_t$. We have $0 \to R_t/\mfq_t^{\lim} \to H^{d-t}_{\fm}(R)$ with \[
    \mfq_t^{\lim} = \dfrac{\bigcup_{n \ge 1} \left(x_1,\ldots,x_t,x_{t+1}^{n+1},\ldots,x_d^{n+1}\right):(x_{t+1}\cdots x_d)^n}{(x_1,\ldots,x_t)}.
    \] 
    We will now induce on $d$. If $d=1$, then \[ (x_1)^{\lim} = (x_1)+\bigcup_n \left(0:_R  x_1^n\right) = (x_1)+ (0:_R x_1) = (x_1)+H^0_\fm(R)\] and hence $(x_1)^{\lim}/(x_1) \cong H^0_\fm(R)$ since $H^0_\fm(R)\cap (x_1) = 0$. Now suppose $d>1$. We have a short exact sequence 
    \begin{center}
        \begin{tikzcd}
            0 \arrow{r} & R/H^0_\fm(R) \arrow{r}{\cdot x_1} & R \arrow{r} & R_1 = R/x_1R \arrow{r} & 0
        \end{tikzcd}
    \end{center} 
    and since $R$ is Buchsbaum,  we have $H^i_\fm(R_1) \cong H^i_\fm(R) \oplus H^{i+1}_\fm(R)$ for all $i < d-1$ (see proof of (iii) $\Rightarrow$ (ii) in \cite[Chapter I, Proposition 2.1]{SV86}). We have another exact sequence
    \begin{equation*}
    \begin{tikzcd}
        0 \arrow{r} & H^{d-1}_\fm(R) \arrow{r} & H^{d-1}_\fm(R_1)\arrow{r} & H^d_\fm(R) \arrow{r}{\cdot x_1} & H^d_\fm(R) \arrow{r} & 0.
    \end{tikzcd}
    \end{equation*}
Furthermore, we have that 
    \begin{equation}\label{eqn: colength formula for lim closure mod a parameter}
        \ell_R(R_1/\mfq_1^{\lim}) = \ell_R(R/\qlim)+\ell_R(H^{d-1}_\fm(R)).
    \end{equation}
    To see this, recall that since $\mfq$ is a parameter ideal of the Buchsbaum ring $R$, we have $$e(\mfq;R)-\ell_R(R/\qlim) = \sum_{i=1}^{d-1} \binom{d-1}{i-1} \ell_R(H^i_\fm(R))$$ by \cite[Theorem 1.1]{CLpseudobuchsbaum}. Then $R_1$ is Buchsbaum since $x_1$ is a parameter element (\cite[Chapter I, Corollary 1.11]{SV86}), and so we also have $e(\mfq_1;R_1)-\ell_R(R_1/\mfq_1^{\lim}) = \sum_{i=1}^{d-2}\binom{d-2}{i-1} \ell_R (H^i_{\fm}(R_1))$. But then, since $e(\mfq;R)=e(\mfq_1;R_1)$, we can use elementary binomial identities to ascertain  \eqref{eqn: colength formula for lim closure mod a parameter}. 

    Letting $\mfq_1' = \cup_n (x_1,x_2^{n+1},\ldots,x_d^{n+1}):(x_2\cdots x_d)^n$ and $R_1 = R/x_1R$, we have $\mfq_1^{\lim} = \mfq_1'/(x_1)$. We can identify $R_1/\mfq_1^{\lim}$ with $R/\mfq_1'$, and hence $\ell_R(\mfq^{\lim}/\mfq_1') = \ell_R(H^{d-1}_{\fm}(R))$. Considering the following diagram 
    \begin{center}
        \begin{tikzcd}
            0 \arrow{r} & \mfq^{\lim}/\mfq_1' \arrow{r}& R/\mfq_1' \arrow{r}\arrow[d,hookrightarrow] & R/\mfq^{\lim} \arrow{r}\arrow[d,hookrightarrow] &0 \\
            0 \arrow{r} & H^{d-1}_\fm(R)\arrow{r}& H^{d-1}_\fm(R_1) \arrow{r} & H^d_\fm(R) 
        \end{tikzcd}
    \end{center}
whose two vertical maps are injective, since $\ell_R(H^{d-1}_{\fm}(R)) = \ell_R(\mfq^{\lim}/\mfq_1')$ we have $\mfq^{\lim}/\mfq_1' \cong H^{d-1}_{\fm}(R)$. 
    Thus, $\qlim/\mfq \cong \mfq_1'/\mfq \oplus H^{d-1}_\fm(R)$ as $k$-vector spaces. By the induction hypothesis, we have \[\mfq_1'/\mfq \cong \bigoplus_{i=0}^{d-2} H^i_\fm(R_1)^{\binom{d-1}{i}}\] as $k$-vector spaces. Moreover, $H^i_{\fm}(R_1) \cong H^i_{\fm}(R) \oplus H^{i+1}_{\fm}(R)$ for all $i \le d-2$, this gives us the first isomorphism. Finally, if $R$ has prime characteristic, then we see that $(\qlim\cap \mfq^F)/\mfq$ is the nilpotent part under the natural Frobenius action, which gives us the second isomorphism in the statement.     
\end{proof}

\section{\texorpdfstring{$F$}{F}-Buchsbaum Rings}\label{sec:CMFI}

Before defining $F$-Buchsbaum rings we will first build some preparatory results in understanding relationships between Frobenius and limit closure of parameter ideals and their connections to local cohomology.

\subsection{Frobenius closure, limit closure, and local cohomology}

Let $(R,\fm)$ be an excellent, equidimensional local ring of prime characteristic $p>0$. Consider the ideal containments $\fq\subseteq \qlim\subseteq \fq^*$ and $\fq\subseteq \fq^F\subseteq \fq^*$. In general, the ideals $\fq^F$ and $\qlim$ are incomparable. In the sequel, we account for this by analyzing their sum, and we denote $\qs:=\fq^F+\fq^{\lim}$ for notational convenience. We also study the Frobenius closure of $\qs$ and build connections between the two ideals under suitable conditions.

\begin{rmk}\label{lem: CMFI iff q=qs}
We can recognize $(\qs)^F$ in another way, specifically $(\qs)^F = (\mfq^{\lim})^F$ for all parameter ideals $\mfq \subseteq R$. To see this, note that $(I+J)^F = (I^F+J^F)^F$ and $(I^F)^F = I^F$ for arbitrary ideals $I,J\subseteq R$, so
\begin{align*}
(\qs)^F=((\mfq^F)^F+(\mfq^{\lim})^F)^F =(\mfq^F+(\mfq^{\lim})^F)^F = (\mfq^{\lim})^F.
\end{align*}
\end{rmk}

The following routine results are aimed at relating $\qs$ and $\qlf$ to the Frobenius closure $0_{H^d_\fm(R)}$. We omit the proof of the first, as it is very similar to that of \cite[Proposition 3.3]{HQ23}.

\begin{lem} \label{lem: computing 0^F using LF}
Let $(R,\fm)$ be a local ring of prime characteristic $p>0$ and of dimension $d>0$. If $\mfq=(x_1,\ldots, x_d)\subseteq R$ is a parameter ideal and $\mfq_j = (x_1^j,\ldots,x_d^j)$ for $j>0$, we have \[
\varinjlim \dfrac{\mfq_j^{{F\textrm{-lim}}}}{\mfq_j^{\lim}} \cong
\varinjlim \dfrac{(\mfq_j^{\lim})^F}{\mfq_j^{\lim}} \cong 0^F_{H^d_\fm(R)}.
\]  
\end{lem}

When $0^F_{H^d_\fm(R)}$ is of finite length, to obtain the isomorphism in the lemma above, we need not take the direct limit. 

\begin{lem}\label{lem: computing 0^F via qs}
Let $(R,\fm)$ be a local ring of dimension $d>0$ and of prime characteristic $p>0$. For all parameter ideals $\fq$ of $R$, we have \[\dfrac{\qs}{\qlim}\hookrightarrow \dfrac{\qlf}{\qlim} \hookrightarrow 0^F_{H^d_\fm(R)}.\] Further, if $0^F_{H^d_\fm(R)}$ is of finite length, then \[0^F_{H^d_\fm(R)}\cong \dfrac{\qs_{n}}{\fq_{n}^{\lim}} \cong \dfrac{(\mfq_n^{\lim})^F}{\mfq_n^{\lim}}\] for all $n \gg 0$.
\end{lem}

\begin{proof}
    For each parameter ideal $\fq$ of $R$, we have a canonical map $R/\fq\rightarrow H^d_\fm(R)$ given by $z+\fq \mapsto [z+\fq]$ since $H^d_\fm(R)=\varinjlim R/\fq_n$. Write $g \colon \qs/\fq \rightarrow H^d_\fm(R)$ for the restriction of this map to $\qs/\fq$. Clearly, we have $\im(g)\subseteq 0^F_{H^d_\fm(R)}$ since \Cref{lem: computing 0^F using LF} tells us that $0^F_{H^d_\fm(R)}\cong\varinjlim \qs_n/\fq_n$. 

    Furthermore, we have $\ker(g) = \qlim/\fq$, as  $z+\fq\in \qlim/\fq$ if and only if $x^nz \in \fq_{n+1}$ for some $n$, where $x$ is the product of a system of parameters which generates $\fq$. But then, we must have $g(z+\fq)=[z+\fq]=0\in H^d_\fm(R)$, as $[z+\fq]=[x^nz+\fq_{n+1}]$. Thus, $$(\qs/\fq)/\ker(g)\cong \qs/\qlim\hookrightarrow 0^F_{H^d_\fm(R)}.$$

    Finally, if $\ell_R(0^F_{H^d_\fm(R)})<\infty$, then $0^F_{H^d_\fm(R)}$ is finitely generated as an $R$-module by elements $\xi_1,\cdots,\xi_t$. We can write $\xi_i=[z_i+\fq_{n(i)}^{\lim}]$, where $z_i \in \qs_{n(i)}$. Letting $n=\max \{n(i) | 1 \le i \le t\}$, we have that $\xi_i=[x^{n-n(i)}z_i + \fq_n^{\lim}]$ by definition of the direct limit system. Then, every generator of $0^F_{H^d_\fm(R)}$ is in the image of the injective map $g:\qs_n/\fq_n^{\lim} \rightarrow 0^F_{H^d_\fm(R)}$, so that $g$ is an isomorphism, as required. The remaining statements about $\qlf$ follow using similar arguments.
\end{proof}

It is well known that limit closure for parameter ideals satisfies the colon-capturing property, in that for any parameter ideal $\mfq=(x_1,\ldots,x_d)$ in an excellent equidimensional local ring, we have \[
(x_1^n,\ldots,x_t^n)^{\lim}:(x_1\cdots x_t)^{n-1} \subseteq (x_1,\ldots,x_t)^{\lim} 
\] for all $n \ge 1$. It is also easy to see that Frobenius closure does not satisfy the above colon-capturing property, e.g. any local ring which is $F$-pure but not Cohen--Macaulay will provide a counterexample, which is a major obstacle in studying versions of this Buchsbaum-like property for Frobenius closure alone. Fortunately, we can show that the operation $(-^{\lim})^F$ satisfies colon capturing.

\begin{thm}\label{thm: colon capturing for qlf}
Let $(R,\fm)$ be a $d$-dimensional local ring of prime characteristic $p>0$, and let $\mfq=(x_1,\ldots,x_d)$ be a parameter ideal of $R$. If $x=x_1\cdots x_d$, we have for all $n \ge 1$ \[
((x_1^n,\ldots,x_d^n)^{\lim})^F:x^{n-1}\subseteq ((x_1,\ldots,x_d)^{\lim})^F.
\]
\end{thm}

\begin{proof}
First note the isomorphism $H^d_\fm(R)\cong \varinjlim R/(x_1^n,\ldots,x_d^n)^{\lim}$ where the direct limit system with injective maps comes from the bottom row of the diagram below. \begin{center}
    \begin{tikzcd}
        R/(x_1^n,\ldots,x_d^n) \arrow{r}{\cdot x} \arrow{d} & R/(x_1^{n+1},\ldots,x_d^{n+1})\arrow{d} \\
        R/(x_1^n,\ldots,x_d^n)^{\lim} \arrow{r}{\cdot x} & R/(x_1^{n+1},\ldots,x_d^{n+1})^{\lim}
    \end{tikzcd}
\end{center}

Now, we have that $\varinjlim ((x_1^n,\ldots,x_d^n)^{\lim})^F/(x_1^n,\ldots,x_d^n)^{\lim}\cong 0^F_{H^d_\fm(R)}$, where the multiplication-by-$x$ maps in the direct limit system are injective. Further, we have $\varinjlim R/((x_1^n,\ldots,x_d^n)^{\lim})^F \cong H^d_\fm(R)/0^F_{H^d_\fm(R)}$, and the direct limit system on the left given by multiplication-by-$x$ is injective, as viewing $R/(x_1^n,\ldots,x_d^n)^{\lim} \hookrightarrow H^d_\fm(R)$, we have $R/(x_1^n,\ldots,x_d^n)^{\lim} \cap 0^F_{H^d_\fm(R)} = ((x_1^n,\ldots,x_d^n)^{\lim})^F/(x_1^n,\ldots,x_d^n)$. But then, injectivity of the direct limit system $\varinjlim R/((x_1^n,\ldots,x_d^n)^{\lim})^F$ implies $(-^{\lim})^F$ satisfies colon capturing, as required. 
\end{proof}

\subsection{The \texorpdfstring{$F$}{F}-Buchsbaum condition}
The Buchsbaum condition sits properly between the notions of Cohen--Macaulayness and generalized Cohen--Macaulayness. Indeed, a local ring $(R,\fm)$ is Cohen--Macaulay if and only if $\ell_R(R/\mfq)-e(\mfq) = 0$ for all parameter ideals $\mfq\subseteq R$, while $R$ is generalized Cohen--Macaulay if and only if $\sup\{\ell_R(R/\mfq)-e(\mfq)\}<\infty$ where the supremum runs over all parameter ideals $\fq\subseteq R$. Similarly, the tight Buchsbaum condition of \cite{MQ22} sits properly between $F$-rationality and the property of having an isolated non-$F$-rational locus.

Recall from \cite[Lemma 2.3]{CHL99} (and the fact that $\qlim \subseteq \qs\subseteq (\qlim)^F$), for any ideal $\fq$ generated by a system of parameters in a Noetherian local ring $(R,\fm),$
\begin{align} \label{chl inequalities} e(\fq) \geq \ell_R(R/\qlim) \geq \ell_R(R/\qs) \geq \ell_R(R/(\qlim)^F). \end{align} This forces the following behavior for $\qs$ and $(\mfq^{\lim})^F$.

\begin{prp}\label{prp:CM}
  Let $(R,\fm,k)$ be a $d$-dimensional reduced, excellent, equidimensional local ring of prime characteristic $p>0$. The following are equivalent. 
  \begin{enumerate}[label=(\alph*)]
      \item $R$ is Cohen-Macaulay and $F$-injective (CMFI).
      \item We have $e(\fq)=\ell_R(R/\qs)$ for all parameter ideals $\fq$.
      \item We have $e(\fq)=\ell_R(R/(\mfq^{\lim})^F)$ for all parameter ideals $\fq$.
  \end{enumerate}
\end{prp}
\begin{proof}
By assumption, the inequalities $\ell_R(R/\qs)\leq\ell_R(R/\qlim)\leq e(\fq)$ are all equalities. By \cite[Corollary 3.3(1)]{CN03} $R$ is Cohen--Macaulay. It follows that $\ell_R(R/\fq^F)=\ell_R(R/\fq)$, hence $R$ is $F$-injective by \cite[Corollary 3.9]{QS17}.
\end{proof}

On the other hand, we can also analyze the condition of being CMFI on the punctured spectrum through the lens of colength and multiplicity. 

\begin{thm}\label{thm: CMFI-pun}(cf. \cite[Theorem 1.1]{GN02})
    Let $(R,\fm)$ be a $d$-dimensional excellent unmixed local ring of prime characteristic $p>0$. The following are equivalent.
    \begin{enumerate}[label=(\alph*)]
        \item For all $\fp\in\Spec(R)\smallsetminus\{\fm\}$, $R_\fp$ is CMFI.\label{thm: CMFI-pun-a}
        
        \item $R$ is generalized Cohen--Macaulay and $\ell_R\left( 0^F_{H^d_\fm(R)}\right)<\infty$.\label{thm: CMFI-pun-b}
        
        \item $\sup\{\ell_R ((\fq^{\lim})^F/\fq)  \mid \fq \text { is a parameter ideal} \} < \infty$.\label{thm: CMFI-pun-c}
        
        \item $\sup\{\ell_R( \qs /\fq)\mid \fq \text{ is a parameter ideal}\}<\infty$.\label{thm: CMFI-pun-c'}
        
        \item $\sup\{ e(\fq) - \ell_R(R/(\fq^{\lim})^F) \mid \fq \text{ is a parameter ideal} \} < \infty$.\label{thm: CMFI-pun-d}
        
        \item $\sup\{ e(\fq) - \ell_R(R/\qs) \mid \fq \text{ is a parameter ideal} \} < \infty$. \label{thm: CMFI-pun-d'}
    \end{enumerate}
\end{thm}

\begin{proof}
We will show that \ref{thm: CMFI-pun-a}$\Rightarrow$\ref{thm: CMFI-pun-b}$\Rightarrow$\ref{thm: CMFI-pun-c}$\Rightarrow$\ref{thm: CMFI-pun-c'}$\Longleftrightarrow$\ref{thm: CMFI-pun-d'}$\Rightarrow$\ref{thm: CMFI-pun-a} and \ref{thm: CMFI-pun-c}$\Longleftrightarrow$\ref{thm: CMFI-pun-d}. We begin by proving the equivalence of \ref{thm: CMFI-pun-c} and \ref{thm: CMFI-pun-d}. Note that $e(\fq) \leq \ell_R(R/\fq)$ implies that $e(\fq) - \ell_R(R/(\fq^{\lim})^F) \leq \ell_R((\fq^{\lim})^F/\fq)$. For the converse, note that  
 \begin{align}
 e(\fq) - \ell_R(R/(\fq^{\lim})^F) = (e(\fq) - \ell_R(R/\qlim)) + \ell_R((\fq^{\lim})^F/\qlim)
 \end{align}
 implies that $\sup_\fq\{e(\fq)-\ell_R(R/\qlim)\}<\infty,$ hence $R$ is generalized Cohen--Macaulay by \cite[Remark 3.5(1)]{MQ22}. As a consequence, $\sup_\fq\{\ell_R(R/\fq) - e(\fq)\}<\infty,$ implying the same for 
 \[ \ell_R((\fq^{\lim})^F/q) = (\ell_R(R/\fq) - e(\fq)) + (e(\fq) - \ell_R(R/(\fq^{\lim})^F)). \] The equivalence of \ref{thm: CMFI-pun-c'} and \ref{thm: CMFI-pun-d'} is similar.

By \cite[Lemma 3.1]{HQ23} we have that \ref{thm: CMFI-pun-a} implies \ref{thm: CMFI-pun-b}. To see that \ref{thm: CMFI-pun-b} implies \ref{thm: CMFI-pun-c}, first note that \[\sup \{\ell_R(\fq^{\lim}/\fq)\} < \infty\] since $(R, \fm)$ is generalized Cohen--Macaulay. From \cref{lem: computing 0^F via qs}, we get $\ell_R((\fq^{\lim})^F/\fq^{\lim}) \le \ell_R(0^F_{H^d_{\fm}(R)})$ since the map in the direct limit is injective. Also, \ref{thm: CMFI-pun-c} implies \ref{thm: CMFI-pun-c'} as \[\ell_R( \qs /\fq) \leq \ell_R ((\fq^{\lim})^F/\fq).\]

We now show that \ref{thm: CMFI-pun-d'} implies \ref{thm: CMFI-pun-a}. Since
\[e(\fq) - \ell_R(R/\qs) = e(\fq) - \ell_R(R/\mfq^{\lim})+\ell_R(\qs/\mfq^{\lim}), \]
we get $\sup_\fq\{e(\fq)-\ell_R(R/\fq^{\lim})\}<\infty$. Hence $R$ is generalized Cohen--Macaulay by \cite{CTS78}. Now let $\fp\in\Spec(R)\smallsetminus\{\fm\}$ be a non-maximal prime. We will show that $R_\fp$ is $F$-injective, and by induction we may assume that $\dim(R_\fp) = d-1$. Let $\fa = (x_1,\dots, x_d) \subseteq R$ be a parameter ideal such that $\fb:=(x_1,\dots,x_{d-1})\subseteq\fp$. Since $\ell_R(\fq^F/\fq) 
\leq \ell_R(\qs/\fq),$ we get $\sup\{\ell_R(\fq^F/\fq)\} < \infty$. Letting $N=\sup_\fq\{\ell_R(\fq^F/\fq)\}$, we see 
\begin{align*}
    \fm^N \fb^F \subseteq \fm^N(\fb+(x_d^n))^F\subseteq \fm^N(\fb+(x_d^n))
\end{align*}
for all $n>0$; hence, $\fm^N\fb^F\subseteq \fb$. Since Frobenius closure commutes with localization, we have
\begin{align*}
    (\fb R_\fp)^F = \fb^F R_\fp = \fb R_\fp.
\end{align*}
We conclude that $R_\fp$ is $F$-injective by \cite[Corollary 3.9]{QS17}.
\end{proof}

Taken together, \cref{prp:CM} and \cref{thm: CMFI-pun}, along with the inspiration from the Buchsbaum and tight Buchsbaum cases, we suggest the following definition.

\begin{dff}
    Let $(R,\fm)$ be an unmixed local ring of dimension $d>0$ and of prime characteristic $p>0$. We say that $R$ is \textit{$F$-Buchsbaum} if the quantity $e(\mfq)-\ell_R(R/\qs)$ is a constant independent of choice of parameter ideal $\fq\subseteq R$.
\end{dff}

\begin{rmk}\label{rmk: f-inj  + buchs = f-buchs}
    Every excellent, $F$-injective generalized Cohen--Macaulay local ring $(R,\fm)$ is $F$-Buchsbaum. Indeed, from \cite[Theorem 1.1]{Ma15}, we know that $\fq=\fq^F$ for all parameter ideals $\fq$ from which we see that $\qs = \fq^{\lim}$. Further from \cite[Theorem~1.2]{Ma15}, it follows that $R$ is Buchsbaum. By \Cref{intro:thm-buchsbaum}, we have that the quantity \[
    e(\mfq)-\ell_R\left(R/\qs\right) = e(\mfq)-\ell_R\left(R/\mfq^{\lim}\right) = e(\mfq)-\ell_R\left(R/\mfq\right)+\ell_R\left(\mfq^{\lim}/\mfq\right)
    \] is a constant independent of $\mfq$.
\end{rmk}

Before we study this condition in detail, we summarize how this new condition relates to both Buchsbaum and tight Buchsbaum rings. Our first task is to see how $F$-Buchsbaum is related to constancy of the length of $\qs/\mfq$ independently of the parameter ideal $\mfq$. We will also show that an $F$-Buchsbaum ring is Buchsbaum.

\begin{thm}[cf. {\cite[Theorem 3.6]{MQ22}}]\label{thm:(a)<=>(b)}
Let $(R,\fm,k)$ be an excellent, unmixed local ring of dimension $d > 0$. The quantity $\ell_R(\qs/\fq)$ is independent of parameter ideal $\fq$ if and only if  $R$ is $F$-Buchsbaum. Moreover, if $R$ is $F$-Buchsbaum then $R$ is Buchsbaum. 

\end{thm}

\begin{proof}
 First assume that $e(\fq)-\ell_R(R/\qs)$ is independent of $\fq$. It follows from 
 \begin{align}
 e(\fq) - \ell_R(R/\qs) = (e(\fq) - \ell_R(R/\qlim)) + \ell_R(\qs/\qlim) \label{eq:(a)<=>(b)-1}
 \end{align}
 that $\sup_\fq\{e(\fq)-\ell_R(R/\qlim)\}<\infty,$ hence $R$ is generalized Cohen--Macaulay by \cite[Remark 3.5(1)]{MQ22}. Setting $$C(R) := \sum\limits_{i=0}^{d-1} \binom{d-1}{i-1}\ell_R(H_\fm^i(R))$$ we see by \cite[Corollary 4.3]{CHL99} and \cite[Theorem 3.7]{SH85} that $e(\fq) - \ell_R(R/\qlim) \leq C(R)$ with equality if and only if $\fq$ is generated by a standard system of parameters \cite[Theorem 5.1]{CHL99}. We also have 
 \begin{align*}
     \ell_R(\qs/\qlim) \leq \ell_R\left(0^F_{H^d_\fm(R)}\right) = \ell_R\left((\fq^{[p^e]})^{F\textrm{-lim}}/(\fq^{[p^e]})^{\lim}\right)
 \end{align*}
for all $e\gg 0$ using Lemma \ref{lem: computing 0^F via qs} as $\sup_\fq\{\ell_R(\qs/\qlim)\}<\infty$ by \eqref{eq:(a)<=>(b)-1}. Using the independence of the quantity $e(\fq) - \ell_R(R/\qs)$ on $\fq$, we get that for all large $e,$
\begin{align*}
     (e(\fq) - \ell_R(R/\qlim)) + \ell_R(\qs/\qlim) 
     &= e(\fq) - \ell_R(R/\qs) \\
     &= e(\fq^{[p^e]}) - \ell_R(R/(\fq^{[p^e]})^{F\textrm{-lim}})  \\
     &= (e(\fq^{[p^e]}) - \ell_R(R/(\fq^{[p^e]})^{\lim})) + \ell_R((\fq^{[p^e]})^{F\textrm{-lim}}/(\fq^{[p^e]})^{\lim})  \\
     &= C(R) + \ell_R(0^F_{H^d_\fm(R)}).
\end{align*}
Thus, both $\ell_R(\qs/\qlim)$ and $e(\fq) - \ell_R(R/\qlim)$ are independent of $\fq.$ Observe then that $R$ is Buchsbaum by \cite[Remark 3.5(2)]{MQ22}. It follows that
$$\ell_R(\qs/\fq) = (\ell_R(R/\fq)-e(\fq)) + (e(\fq) - \ell_R(R/\qs))$$ is independent of $\fq$ as claimed.

If $\ell_R(\qs/\fq)$ is constant for all parameter ideals $\fq$ then $R$ is generalized Cohen--Macaulay. By \cref{thm: CMFI-pun}, $\ell_R(\zdF)<\infty$. Write $\ell_R(\qs/\fq) = \ell_R(\qs/\qlim) + \ell_R(\qlim/\fq)$. By \cite[Remark 2.4(3) and Lemma 3.3]{MQ22}, $$\ell_R(\qlim/\fq) \leq \sum_{i=0}^{d-1} \binom{d}{i} \ell_R(H^i_\fm(R)) = \ell_R\left((\fq^{[p^e]})^{\lim} / (\fq^{[p^e]})\right)$$ for all $e\gg 0$ since $\fq^{[p^e]}$ is standard. \cref{lem: computing 0^F via qs} shows that for all $e\gg 0$, $$\ell_R \left( \qs / \qlim \right) \leq \ell_R(0^F_{H^d_\fm(R)}) = \ell_R\left( (\fq^{[p^e]})^{F\textrm{-lim}} / (\fq^{[p^e]})^{\lim} \right).$$ 
Since $\ell_R(\qs/\fq)$ is constant, we observe that
\begin{align*}
\ell_R(\qs/\qlim) + \ell_R(\qlim/\fq) =& \ell_R(\qs/\fq)\\
=&   \ell_R((\fq^{[p^e]})^{F\textrm{-lim}}/\fq^{[p^e]})  \\ 
=& \ell_R((\fq^{[p^e]})^{F\textrm{-lim}}/(\fq^{[p^e]})^{\lim}) + \ell_R((\fq^{[p^e]})^{\lim}/\fq^{[p^e]}).
\end{align*}
Thus both $\ell_R(\qs/\qlim)$ and $\ell_R(\qlim/\fq)$ are independent of $\fq.$ In particular, $R$ is Buchsbaum by \cite[Remark 3.4]{MQ22} and hence $e(\fq) - \ell_R(R/\qs) = e(\fq) - \ell_R(R/\fq) + \ell_R(\qs/\fq)$ is independent of $\fq$. 
\end{proof}

Now we show that tight Buchsbaum rings are $F$-Buchsbaum.

\begin{prp}\label{prp: tbb implies fbb}
Suppose $(R,\fm,k)$ is an excellent, unmixed local ring of dimension $d>0$. If $R$ is tight Buchsbaum, then $R$ is also $F$-Buchsbaum.
\end{prp}

\begin{proof}
    For any parameter ideal $\fq,$ we have the short exact sequence of $k$-vector spaces 
    \begin{center}
        \begin{tikzcd}
            0 \arrow{r} & \fq^F/\fq \arrow{r} & \fq^*/\fq \arrow{r} & \fq^*/\fq^F \arrow{r} & 0.  
        \end{tikzcd}
    \end{center}
    Since $R$ is tight Buchsbaum, from \Cref{lem: buchs imples qlim/q iso to lc}, together with the proof of \cite[Theorem.~4.3]{MQ22}, we have an isomorphism of $k$-vector spaces \[ \dfrac{\fq^*}{\fq} \cong \bigoplus_{i=0}^{d-1} H^i_\fm(R)^{\binom{d}{i}} \oplus 0^*_{H^d_\fm(R)}  .
    \] By restricting to the part which vanishes under the relative Frobenius action, it follows that \[
    \dfrac{\fq^F}{\fq} \cong \bigoplus_{i=0}^{d-1} 0^F_{H^i_\fm(R)}{}^{\binom{d}{i}} \oplus 0^F_{H^d_\fm(R)}.
    \] This shows that if $R$ is tight Buchsbaum, then the quantity \[\ell_R(\fq^F/\fq)  = \sum_{i=0}^{d-1} \binom{d}{i} \ell_R(0^F_{H^i_\fm(R)}) + \ell_R(0^F_{H^d_\fm(R)})\] is independent of parameter ideal $\fq$. By \cite[Remark 3.7]{MQ22}, $R$ is Buchsbaum so $\ell_R(\qlim/\fq)$ is also a constant independent of $\fq$. Similarly, $\ell_R(\qlim\cap\fq^F/\fq)$ is independent of $\fq$ by \Cref{lem: buchs imples qlim/q iso to lc}. We then see that \[
    \ell_R(\qs/\fq)=\ell_R(\fq^F/\fq)+\ell_R(\qlim/\fq)-\ell_R(\qlim\cap\fq^F/\fq)
    \] is a constant independent of $\fq$, so $R$ is $F$-Buchsbaum by \Cref{thm:(a)<=>(b)}.
    \end{proof}

If $R$ is $F$-Buchsbaum then the ideals $\qs$ are Frobenius closed, i.e. $\qs=\qlf$ for all $\fq$. In this case, the length $\ell_R(\mfq^F/\mfq)$ does not depend on $\fq$ either.

\begin{thm}\label{lem: qs/q constant length implies qF/q constant length}
Let $(R,\fm,k)$ be an excellent, unmixed local ring of prime characteristic $p>0$, and of dimension $d>0$. If $R$ is $F$-Buchsbaum, then:
\begin{enumerate}[label=(\roman*)]
    \item $\qlf = \qs$ for each parameter ideal $\fq$ of $R$;
    \item $\ell_R(\mfq^F/\mfq)$ is a constant independent of parameter ideal $\mfq$;
    \item $e(\fq)-\ell_R(R/\fq^F)$ is a constant independent of  parameter ideal $\fq$.\label{thm:q-f-lim-3}
\end{enumerate}
\end{thm}

\begin{proof}
$R$ is Buchsbaum and $R$ is $F$-injective on the punctured spectrum by \Cref{thm:(a)<=>(b),thm: CMFI-pun}, so $0^F_{H^j_\fm(R)}$ are finite length modules for all $0 \le j \le d$. By hypothesis, there is a constant $D$ such that $\ell_R(\qs/\mfq)=D$ for all parameter ideals $\mfq$ of $R$. Similarly, by \Cref{intro:thm-buchsbaum}, there is a constant $C$ such that $\ell_R(\qlim/\mfq)=C$ for all parameter ideals $\mfq$ of $R$. The natural short exact sequence yields $\ell_R(\qs/\qlim) = D-C$ for all parameter ideals $\mfq$. By \Cref{lem: computing 0^F via qs}, we have $\ell_R\left(0^F_{H^d_\fm(R)}\right)=D-C$ as well. Then, from the inclusions \[
\qs/\qlim \subseteq \qlf/\qlim \subseteq 0^F_{H^d_\fm(R)},
\] we have that $\qs=\qlf$ for all parameter ideals $\mfq$. 

Now, from the diagram
\begin{center}
\begin{tikzcd}
    0 \arrow{r} & \dfrac{\mfq^F\cap\qlim}{\mfq} \arrow{r} & \dfrac{\mfq^F}{\mfq} \arrow{r} & \dfrac{\mfq^F}{\mfq^F\cap \qlim}\arrow[d,"\cong"] \arrow{r} & 0 \\
    0 \arrow{r} & \dfrac{\qlim}{\mfq} \arrow{r} & \dfrac{\qs}{\mfq} \arrow{r} & \dfrac{\qs}{\qlim} \arrow{r} & 0
    \end{tikzcd}
\end{center} with exact rows we see \[
\ell_R\left( \dfrac{\mfq^F}{\mfq} \right) = \ell_R\left( \dfrac{\qs}{\mfq} \right) - \ell_R\left( \dfrac{\qlim}{\mfq} \right) + \ell_R\left( \dfrac{\mfq^F\cap \qlim}{\mfq} \right). 
\] Each term on the right-hand side is independent of $\mfq$ by hypothesis, \Cref{intro:thm-buchsbaum}, and \Cref{lem: buchs imples qlim/q iso to lc}. Finally, \ref{thm:q-f-lim-3} follows from the constancy of $\ell_R(R/\fq)-e(\fq)$.
\end{proof}

Another setting in which we have $\qs=\qlf$ is if we assume $R$ is either $F$-nilpotent, or Buchsbaum and weakly $F$-nilpotent.

\begin{rmk}
Suppose $R$ is an excellent, equidimensional local ring of prime characteristic $p>0$. If $R$ is $F$-nilpotent, we have $\fq^* = \fq^F$ for all parameter ideals $\mfq$ by \cite[Thm. 5.11]{PQ19}. Since $\fq^{\lim} \subseteq \fq^*$, we indeed have $\qs = \fq^F = (\fq^{\lim})^F$. 
\end{rmk}

If we assume $R$ is Buchsbaum, we can replace $F$-nilpotence with weak $F$-nilpotence.

\begin{prp}\label{prop:wfn-q-lim-F}
  Let $(R, \fm)$ be a Buchsbaum, weakly $F$-nilpotent local ring of dimension $d>0$. For every parameter ideal $\fq\subseteq R$, we have $\fq^{\lim} \subseteq \fq^F$, and consequently, we have the equality $\qs = \mfq^F = (\mfq^{\lim})^F$.
\end{prp}

\begin{proof}
Let $\fq = (x_1, \ldots, x_d)$ be a parameter ideal. We have 
$$\fq^{\lim} = \fq + \sum_{i=1}^n (x_1, \ldots, \widehat{x}_i, \ldots, x_d) :_R x_i.$$
For each $i \ge 1$, we have 
$$\frac{(x_1, \ldots, \widehat{x}_i, \ldots, x_d) :_R x_i}{(x_1, \ldots, \widehat{x}_i, \ldots, x_d)} \cong \bigoplus_{j=0}^{d-1} H^j_{\fm}(R)^{\binom{d-1}{j}}.$$ 
As $R$ is weakly $F$-nilpotent, the relative Frobenius action of the left-hand side is also nilpotent. So
$$(x_1, \ldots, \widehat{x}_i, \ldots, x_d) :_R x_i \subseteq (x_1, \ldots,\widehat{x}_i, \ldots, x_d)^F.$$
Hence $\fq^{\lim} \subseteq \fq^F$. The last assertion is clear.
\end{proof}

\cref{prop:wfn-q-lim-F} and \Cref{lem: qs/q constant length implies qF/q constant length} prompt the following question.

\begin{qst}
Let $(R,\fm)$ be an equidimensional local ring of dimension $d>0$ and of prime characteristic $p>0$. Under what hypotheses on $R$ does the equality $\qs=(\qlim)^F$ hold for all parameter ideals $\fq\subseteq R$?
\end{qst}

\section{Proof of the main theorem}\label{sec:main}

The goal of this section is to prove \cref{thm:mainthm} stated in the introduction, which we restate below for convenience of the reader. Recall that if the excellent, equidimensional local ring $(R,\fm)$ is $F$-Buchsbaum, then $R$ is Buchsbaum and $\qs = (\fq^{\lim})^F$ for all parameter ideals $\fq$ by \cref{thm:(a)<=>(b)} and \Cref{lem: qs/q constant length implies qF/q constant length}.

\begin{thm}\label{thm:mainthm-restated}
For an excellent $d$-dimensional unmixed local ring $(R,\fm,k)$ of prime characteristic $p>0$, consider the following conditions:
\begin{enumerate}[label=(\alph*)]
\item $R$ is $F$-Buchsbaum.\label{main-thm-0}
\item\label{a} $\fm (\qs) \subseteq \fq$ for all $\fq$ generated by system of parameters. 
\item\label{b} the $F$-truncation $\tau^{<d,F}\mathbf{R} \Gamma_\fm(R)$ is quasi-isomorphic to a complex of $k$-vector spaces.
\end{enumerate}
If $R$ is $F$-finite then \ref{main-thm-0} implies \ref{a}. If $R$ is $F$-finite and weakly $F$-nilpotent then \ref{main-thm-0} implies \ref{b}. If $R$ is weakly $F$-nilpotent and has a dualizing complex, then \ref{b} $\Rightarrow$ \ref{a} $\Rightarrow$ \ref{main-thm-0}. In particular, if $R$ is weakly $F$-nilpotent and $F$-finite then \ref{main-thm-0}--\ref{b} are equivalent. Finally, if $R$ is Cohen--Macaulay and $F$-finite, then \ref{main-thm-0}--\ref{b} are equivalent to:
\begin{enumerate}[label=(\alph*)]
\setcounter{enumi}{3}
    \item \label{c} there exists a parameter ideal $\mathfrak{Q} \subseteq \bigcap_{\fq} (\fq:\fq^F)$ (where the intersection ranges over all parameter ideals $\fq\subseteq R$) such that $\mathfrak{Q}:\mathfrak{Q}^F\supseteq\fm$. 
\end{enumerate}
\end{thm}

The first claim is proven in \cref{thm:(b)=>(d)} and the second claim is proven in \cref{thm:(d)=>(c),thm:(c)=>(b)WNil}. The last claim is proven in \cref{cor:gor}.

\subsection{Properties of \texorpdfstring{$F$}{F}-Buchsbaum rings}

\begin{thm}[cf. {\cite[Theorem 5.1]{MQ22}}]\label{thm:(b)=>(d)} Let $(R,\fm,k)$ be an $F$-finite local ring of dimension $d > 0$. If $R$ is $F$-Buchsbaum, then:
\begin{enumerate}[label=(\roman*)]    
    \item $\fm\qs\subseteq\fq$ for all parameter ideals $\fq$, and \label{thm:(b)=>(d)-conclusion-2}
    \item if we further assume that $R$ is weakly $F$-nilpotent, then the $F$-truncation $\tau^{<d,F}\mathbf{R}\Gamma_\fm(R)$ is quasi-isomorphic to a complex of $k$-vector spaces.\label{thm:(b)=>(d)-conclusion-3}
\end{enumerate}
\end{thm}
\begin{proof}
 By \Cref{thm: CMFI-pun}, $R_\fp$ is Cohen--Macaulay and $F$-injective for all $\fp\in\Spec^\circ(R)$. In particular, $R_\fp$ is regular for all minimal primes $\fp$, hence $R$ is reduced and equidimensional; note that we are using here that $F$-Buchsbaum rings are unmixed by assumption. Note that $R$ is Buchsbaum by \cref{thm:(a)<=>(b)}. By \cite{HKSY06} we may fix an integer $e\gg \hsl(R)$ so that 
 \begin{align}
 \fq^F/\fq = \ker(R/\fq \to F^e_*R/\fq F^e_*R\cong F^e_*(R/\fq^{[p^e]}))\label{eq:(b)=>(d)-kernel}
 \end{align}
 for all parameter ideals $\fq\subseteq R$. For this choice of $e$ and the exact sequence
 \begin{align}
 0\to R\to F^e_* R\to C\to 0\label{eq:(b)=>(d)-2}
 \end{align}
we claim that $F^e_* R$ and $C=F^e_* R/R$ are Buchsbaum $R$-modules. Indeed, the quantity
\begin{align*}
    \ell_R (F^e_*R/\fq F^e_* R) - e(\fq,F^e_*R) = \left[k^{1/p^e}:k\right]\left(\ell_R\left(R/\fq^{[p^e]}\right) - e\left(\fq^{[p^e]}\right)\right)
\end{align*}
is independent of parameter ideal $\fq$ so that $F^e_*R$ is Buchsbaum. If we tensor \eqref{eq:(b)=>(d)-2} with $(-)\otimes_R R/\fq$ we obtain via \eqref{eq:(b)=>(d)-kernel} the exact sequence
 \begin{align}
 0\to \fq^F/\fq\to R/\fq\to F^e_*(R/\fq^{[p^e]})\to C/\fq C\to 0.\label{eq:(b)=>(d)-3}
 \end{align} To see that $C$ is Buchsbaum, from \eqref{eq:(b)=>(d)-2} and \eqref{eq:(b)=>(d)-3} we find that
 \begin{eqnarray}
\ell_R(C/\fq C) - e(\fq, C) & = & \ell_R(C/\fq C) - e(\fq,F_*^e R) + e(\fq)\nonumber\\ 
 & = & \ell_R(F_*^e R / \fq F_*^e R) - \ell_R(R/\fq) + \ell_R(\fq^F/\fq) - e(\fq,F_*^e R) + e(\fq)\nonumber\\ 
 & = & \ell_R(\fq^F/\fq) - (\ell_R(R/\fq) - e(\fq)) + (\ell_R(F_*^e R / \fq F_*^e R) - e(\fq, F_*^e R))\label{eq:(b)=>(d)-7}
\end{eqnarray}
is independent of $\fq$ as desired, applying \cref{lem: qs/q constant length implies qF/q constant length} to \eqref{eq:(b)=>(d)-7}. Letting $\fq = (x_1,\dots, x_d)$ and examining the long exact sequence on Koszul homology induced by \eqref{eq:(b)=>(d)-2} gives a surjection $H_1(\ux; C)\twoheadrightarrow \fq^F /\fq$. Since $C$ is Buchsbaum, we observe that $H_1(\ux; C)$ (and hence $\fq^F /\fq$) is annihilated by $\fm$. In light of \cref{lem: buchs implies m q lim in q}, this completes \ref{thm:(b)=>(d)-conclusion-2}.

From \eqref{eq:(b)=>(d)-2} we have an exact triangle
\begin{align*}
 \R\Gamma_\fm C[-1] \to \R\Gamma_\fm R \to \R\Gamma_\fm F_*^e R \xra{+1}.
 \end{align*}
Since we are assuming that $R$ is weakly $F$-nilpotent and since we picked $e \gg 0$ larger than the HSL-number of $R$, we have exact sequences
\begin{align}
0\to H^{i-1}_\fm(F^e_* R)\to h^{i}(\R\Gamma_\fm C[-1]) \to H_\fm^{i}(R)\to 0\label{eq:(b)=>(d)-4}
\end{align} for all $i < d$ and 
\begin{align}
    0\to H_\fm^{d-1}(F_*^eR) \to h^d(\R\Gamma_\fm C[-1]) \to H^d_\fm(R)\to H^d_\fm(F^e_* R).\label{eq:(b)=>(d)-8}
\end{align}
By our choice of $e\gg 0$, note that $\ker( H_\fm^d(R) \to H_\fm^d(F_*^e R)) = 0^F_{H^d_\fm(R)}$ so that \eqref{eq:(b)=>(d)-8} induces
\begin{align}
0\to H_\fm^{d-1}(F_*^eR) \to h^d(\R\Gamma_\fm C[-1]) \to 0^F_{H^d_\fm(R)} \to 0.\label{eq:(b)=>(d)-5}
\end{align}
Note that $\tau^{\leq d}(\R\Gamma_\fm C[-1])$ lives in cohomological degree $[0,d]$ while $H^d_\fm(R)/0^F_{H^d_\fm(R)}[-d]$ lives strictly in cohomological degree $d$. By the above, we also see that $h^d(\R\Gamma_\fm C[-1]) \to H^d_\fm(R)\to H^d_\fm(R)/0^F_{H^d_\fm(R)}$ is the zero map. Assembling everything, we therefore have
\begin{equation}
    \begin{tikzcd}
        \tau^{<d,F}\R\Gamma_\fm R\arrow[r]& \R\Gamma_\fm R\arrow[r,"\cong"]&\tau^{\leq d}\R\Gamma_\fm R\arrow[r]& H^d_\fm(R)/0^F_{H^d_\fm(R)}[-d]\arrow[r,"+1"]&{}\\
        & & \tau^{\leq d}(\R\Gamma_\fm C[-1])\arrow[u]\arrow[ur,"0"]& &
    \end{tikzcd}\label{eq:(b)=>(d)-6}
\end{equation}
This constitutes the data of a map $\tau^{\leq d} \R\Gamma_\fm C[-1] \to \tau^{<d,F}\R\Gamma_\fm(R)$ with induced surjections
\begin{align*}
    h^i(\tau^{\leq d}(\R\Gamma_\fm C[-1]))\cong h^i(\R\Gamma_\fm C[-1])\twoheadrightarrow H^i_\fm(R)&\hspace{1cm}\text{by \eqref{eq:(b)=>(d)-4}, and }\\
    h^d(\tau^{\leq d}(\R\Gamma_\fm C[-1]))\cong h^d(\R\Gamma_\fm C[-1])\twoheadrightarrow 0^F_{H^d_\fm(R)}&\hspace{1cm}\text{by \eqref{eq:(b)=>(d)-5},}
\end{align*}
for all $i<d$. Our final goal is to precompose this map with a map from an object of $D(k)$ likewise inducing surjections on cohomology. By Gabber's observation \cite[Remark 13.6]{Gab04} we may write $R\cong A/I$ where $A$ is a regular local ring of characteristic $p>0$. Then we have a quasi-isomorphism  $K^\bullet(\fm,C)\qis \R \Hom_A(k,C)$ where $K^\bullet(\fm,C)$ is the (cohomological) Koszul complex on the generators of $\fm$. Moreover, we see that the map
$K^\bullet(\fm,C)[-1]\to \R\Gamma_\fm C[-1]$ induces surjective maps on cohomology by \cite[Theorem 3.4]{Tru86}. Composing this map with the one provided by \eqref{eq:(b)=>(d)-6} gives 
$$K^\bullet(\fm,C)[-1]\to \tau^{<d,F}\R\Gamma_\fm(R)$$ inducing surjections $h^i(K^\bullet(\fm,C)[-1])\twoheadrightarrow h^i(\tau^{<d,F}\R\Gamma_\fm(R))$. It follows by \cite[II, Proposition 4.3]{SV86} that $\tau^{<d,F}\R\Gamma_\fm(R)\in D(k)$ as desired.
\end{proof}

\begin{rmk}
    It would be natural to employ the $\Gamma$-construction of \cite{HH94} to remove the $F$-finiteness assumption from \cref{thm:(b)=>(d)}. We briefly summarize why we have not done so. In the setting of \cref{thm:(b)=>(d)}, the completion $\widehat{R}$ will retain all relevant assumptions about $R$. We may choose a $p$-basis $\Lambda$ of $k$ and a sufficiently small cofinite subset $\Gamma\ll \Lambda$ so that $\widehat{R}^\Gamma_\fp$ is Cohen--Macaulay and $F$-injective for all $\fp\in\Spec^\circ(\widehat{R}^\Gamma)$. By the proof of \cite[Lemma 2.9(a)]{EH08} we can also shrink $\Gamma$ so that $0^F_{H^i_\fm(R)}\otimes_R R^\Gamma\cong 0^F_{H^i_{\fn}(R^\Gamma)}$ for all $i\leq d$. This will allow us to show using the methods of \cite{HQ23}, as in the proof of \cite[Lemma 4.4]{Quy18} in the tight closure setting, that $(\fq^{[q]})^F \widehat{R}^\Gamma= (\fq^{[q]}\widehat{R}^\Gamma)^F$ for all parameter ideals $\fq\subseteq \widehat{R}$ and for all $q=p^e\gg 0$. However, we do not see how to conclude that $\fq^F\widehat{R}^\Gamma = (\fq \widehat{R}^\Gamma)^F$, as the analogous step for tight closure uses colon capturing.
\end{rmk}

\subsection{Interactions between the \texorpdfstring{$F$}{F}-Buchsbaum and weakly \texorpdfstring{$F$}{F}-nilpotent conditions} Recall by \Cref{prop:wfn-q-lim-F} that parameter ideals $\fq\subseteq R$ in a weakly $F$-nilpotent Buchsbaum local ring $R$ enjoy the equality $$\qs = (\fq^{\lim})^F = \fq^F.$$ We need a reformulation of the $F$-truncation $\tau^{< d,F} \mathbf{R} \Gamma_{\fm} R$ in terms of $\rperf$.

\begin{lem}[cf. {\cite[Lemma 2.9]{MQ22}}]\label{lem:wfn-BCM-qis}For $(R,\fm)$ an excellent, reduced, equidimensional, weakly $F$-nilpotent local ring of dimension $d$ and of prime characteristic $p>0$, $$\tau^{\leq d}( \mathbf{R} \Gamma_{\fm}(R_\perf/R)[-1])\cong\tau^{< d,F} \mathbf{R} \Gamma_{\fm} R$$ in $D(R)$.
\end{lem}
\begin{proof}
The short exact sequence $$0 \to R \to \rperf \to \rperf/R \to 0$$ of $R$-modules induces the triangle
\begin{align}
    \mathbf{R}\Gamma_\fm(\rperf/R)[-1]\to \mathbf{R}\Gamma_\fm R\to\mathbf{R}\Gamma_\fm\rperf \stackrel{+1}{\to}.\label{eq:MQ2.9-1}
\end{align}
By \cref{thm:wfn-rperf} since $R$ is weakly $F$-nilpotent, we have $H^i_\fm(\rperf) = 0$ for all $i<d$. Taking cohomology we see then for $i<d$ that
$$h^i(\mathbf{R} \Gamma_{\fm} (\rperf/R)[-1]) \cong h^i(\mathbf{R} \Gamma_{\fm} R)=H^i_\fm(R).$$ For $i= d$ we have 
$$h^d(\mathbf{R} \Gamma_{\fm} (\rperf/R)[-1]) \cong \ker(H^d_\fm(R)\to H^d_\fm(\rperf))\cong 0^F_{H^d_\fm(R)}.$$  Composing with the induced map we have
\begin{align}
\tau^{\leq d} (\mathbf{R} \Gamma_{\fm} (\rperf/R)[-1]) \to \tau^{\leq d} \mathbf{R} \Gamma_{\fm} R\cong \mathbf{R} \Gamma_{\fm} R \to H_{\fm}^d(R)/0^F_{H^d_\fm(R)}[-d]\label{R-perf-truncation-SES}
\end{align}
whose $d$\ts{th} cohomology vanishes. By cohomological degree considerations, \eqref{R-perf-truncation-SES} identifies with the $0$-map in $D(R)$, hence the induced map $\tau^{\leq d} (\mathbf{R} \Gamma_{\fm} (\rperf/R)[-1])\to \tau^{<d,F}\mathbf{R}\Gamma_\fm R$ is a quasi-isomorphism as claimed.
\end{proof}

\begin{thm}\label{thm:(d)=>(c)}
Let $(R,\fm,k)$ be an ummixed, excellent, weakly $F$-nilpotent local ring of dimension $d > 0$. Further assume that $R$ has a dualizing complex. If $\tau^{<d,F} \mathbf{R}\Gamma_{\fm}(R)$ is quasi-isomorphic to a complex of $k$-vector spaces, then $\fm (\fq^{\lim})^F \subseteq \fq$ for all parameter ideals $\fq\subseteq R$. 
\end{thm}

\begin{proof} 
By assumption, $\ell_R(H^i_\fm(R))<\infty$ for all $i<d$, and $\ell_R(\zdF)<\infty$ so $R_\fp$ is $F$-injective for all $\fp\in \Spec(R)\smallsetminus\{\fm\}$ by \cref{thm: CMFI-pun}. In particular, $R_\fp$ is regular for all minimal primes $\fp$; hence, $R$ is reduced and equidimensional. Let $\uline{x}:=x_1,\dots, x_d$ be a system of parameters, write $\fq=(\uline{x})$ and let $K_\bullet(\uline{x},R)$ denote the (homological) Koszul complex on $\uline{x}$. The short exact sequence $$0 \to R \to \rperf \to \rperf/R \to 0$$ induces a triangle as in \eqref{eq:MQ2.9-1} which we then tensor with $K_\bullet(\uline{x},R)$ to obtain the triangle of $R$-modules below.
\begin{equation}
    \begin{tikzcd}
        K_\bullet(\uline{x},R)\otimes \mathbf{R}\Gamma_\fm(\rperf/R)[-1]\arrow[r]\arrow[d,"\qis"]& K_\bullet(\uline{x},R)\otimes\mathbf{R}\Gamma_\fm R\arrow[d,"\qis"]\arrow[r]& K_\bullet(\uline{x},R)\otimes \mathbf{R}\Gamma_\fm\rperf\arrow[d,"\qis"]\arrow[r,"+1"]& {}\\
        K_\bullet(\uline{x},\rperf/R)[-1]\arrow[r]&K_\bullet(\uline{x},R)\arrow[r]& K_\bullet(\uline{x},\rperf)\arrow[r,"+1"]& {}
    \end{tikzcd}
\end{equation}
The bottom triangle induces in cohomology the exact sequence 
\begin{align*}
    0=h^{-1} (K_{\sbt}(\ux,\rperf)) \to h^0(K_{\sbt}(\ux,\rperf/R)[-1]) \to R/\fq \to \rperf/\fq \rperf,
\end{align*}
where the leftmost zero follows from \cref{thm:wfn-rperf} using the assumption that $R$ is weakly $F$-nilpotent and hence $\underline{x}$ is a regular sequence on $\rperf$. We then see that 
\begin{equation}
    h^0(K_\bullet(\uline{x},R)\otimes \mathbf{R}\Gamma_\fm(\rperf/R)[-1])\cong h^0( K_\bullet(\uline{x},\rperf/R)[-1])\cong \ker(R/\fq\to \rperf/\fq\rperf)\cong \fq^F/\fq.\label{eq:QF-mod-Q-isomorphism}
\end{equation}
Applying $K_\bullet(\uline{x},R)\otimes(-)$ to the canonical triangle
\begin{equation}
\begin{tikzcd}
    \tau^{\leq d}(\R \Gamma_\fm(\rperf/R)[-1])\arrow[r]& \R \Gamma_\fm(\rperf/R)[-1]\arrow[r]& \tau^{>d}(\R \Gamma_\fm(\rperf/R)[-1])\arrow[r,"+1"]&{}
\end{tikzcd}\label{eq:(d)=>(c)-2}
\end{equation}
gives the triangle
{\footnotesize
\begin{align}
    K_\bullet(\uline{x},R)\otimes\left(\tau^{\leq d}(\R \Gamma_\fm(\rperf/R)[-1])\right)\to K_\bullet(\uline{x},R)\otimes\left(\R \Gamma_\fm(\rperf/R)[-1]\right)\to K_\bullet(\uline{x},R)\otimes\left(\tau^{>d}(\R \Gamma_\fm(\rperf/R)[-1])\right)\stackrel{+1}{\to}.\label{eq:(d)=>(c)-3}
\end{align}
}
Taking cohomology of \eqref{eq:(d)=>(c)-3} yields the isomorphism
\begin{align}
    h^0 \big( K_{\sbt}(\ux,R) \otimes \tau^{\leq d}(\R \Gamma_\fm(\rperf/R)[-1]) \big) &\cong h^0 \big( K_\bullet(\uline{x},R)\otimes \mathbf{R}\Gamma_\fm(\rperf/R)[-1] \big) \label{eq:(d)=>(c)-4}\\
    & \cong h^0(K_{\sbt}(\ux,\rperf/R)[-1])
\end{align}
because the third term in \eqref{eq:(d)=>(c)-2} lives in cohomology degree $>d$ while the complex $K_\bullet(\uline{x},R)$ of finite free $R$-modules lives in cohomology degrees $[-d,0]$, and hence
\begin{align*}
    h^0\left(K_\bullet(\uline{x},R)\otimes\tau^{>d}(\R \Gamma_\fm(\rperf/R)[-1])\right) = h^{-1}\left(K_\bullet(\uline{x},R)\otimes\tau^{>d}(\R \Gamma_\fm(\rperf/R)[-1])\right) = 0.
\end{align*}
Plugging the left hand side of \eqref{eq:(d)=>(c)-4} into \eqref{eq:QF-mod-Q-isomorphism} yields the isomorphism
\begin{align}
    h^0( K_{\sbt}(\ux,R) \otimes \tau^{\leq d}(\R \Gamma_\fm(\rperf/R)[-1]))\cong \fq^F/\fq.\label{eq:(d)=>(c)-5}
\end{align}
By hypothesis and \cref{lem:wfn-BCM-qis} we have that $\tau^{\leq d}(\R \Gamma_\fm(\rperf/R)[-1])$ is quasi-isomorphic to a complex of $k$-vector spaces. Therefore $h^0( K_{\sbt}(\ux,R) \otimes \tau^{\leq d}(\R \Gamma_\fm(\rperf/R)[-1]))$ is a $k$-vector space and $\fq^F/\fq$ is annihilated by $\fm$ for every parameter ideal $\fq$. Since $R$ is Buchsbaum by \cref{intro:thm-buchsbaum}, it follows from \cref{lem: buchs implies m q lim in q,prop:wfn-q-lim-F} that $\fm(\fq^{\lim})^F=\fm\qs\subseteq\fq$ finishing the proof.
\end{proof}

\begin{thm}\label{thm:(c)=>(b)WNil} If $R$ is weakly $F$-nilpotent and $\fm \qs \subseteq \fq$ for every parameter ideal $\fq$, then the quantity $\ell_R(\qs/\fq)$ does not depend on parameter ideal $\fq \subseteq R$.    
\end{thm}

\begin{proof}

Since $\fm\fq^{\lim}\subseteq \fm \qs \subseteq \fq$ for all parameter ideals $\fq$, it follows that $R$ is Buchsbaum by \cref{lem: buchs implies m q lim in q}. Hence, all parameter ideals are standard, i.e., $\ell_R(\qlim/\fq)$ is constant. To prove the theorem, it suffices to show that $\qs/\fq^{\lim} \to 0^F_{H^d_\fm(R)}$ is surjective for all parameter ideals $\fq$. Indeed, first observe that $$\ell_R(\qs/\fq) = \ell_R(\qlim/\fq) + \ell_R(\qs/\qlim)$$ so it is enough to show that $\ell_R(\qs/\qlim)$ is constant. In light of \Cref{lem: computing 0^F via qs}, if $\qs/\fq^{\lim} \to 0^F_{H^d_\fm(R)}$ is surjective, then it will follow that $\ell_R(\qs/\qlim) = \ell_R(0^F_{H^d_\fm(R)})$. 

Note further the surjection $\fq^F/\fq \twoheadrightarrow \qs/\fq^{\lim} \cong \fq^F/(\fq^F \cap \fq^{\lim})$ and that the map $\beta:\fq^F/\fq \rightarrow 0^F_{H^d_\fm(R)}$ factors through $\qs/\fq^{\lim}.$ Thus, it suffices to show that $\beta$ is surjective.

Now fix  $\fq = (x_1,\ldots,x_d)$ and denote $\qi = (x_1,\ldots,x_i)$. By \cite[Lemma 5.5]{PQ19} 
\begin{equation}\label{eq:0rho}
    0^\rho_{H_\fm^{d-i}(R/\qi)} \cong \varinjlim_n \frac{(\qi,x_{i+1}^n,\ldots,x_d^n)^F}{(\qi,x_{i+1}^n,\ldots,x_d^n)}.
\end{equation} It is clear by definition when $i = d$, $0^\rho_{H_\fm^{d-i}(R/\qi)} = 0^\rho_{H_\fm^{0}(R/\fq)} \cong \fq^F/\fq$ and when $i = 0$ we have $0^\rho_{H_\fm^{d}(R)} = 0^F_{H^d_\fm(R)}$. As $(\qi,x_{i+1}^n,\ldots,x_d^n)$ is a parameter ideal for all $n$, we know by assumption that $x_{i+1}0^\rho_{H_\fm^{d-i}(R/\qi)} = 0$ for all $i$. We will check that this induces surjections $0^\rho_{H_\fm^{d-i-1}(R/\qin)} \twoheadrightarrow 0^\rho_{H_\fm^{d-i}(R/\qi)}$ for all $i$. 

For this, \cite[Lemma 4.2]{MQ22} provides an exact sequence 
$$\cdots \to H_\fm^{d-i-1}(R/\qi) \xra{\alpha} 0^{*_\rho}_{H_\fm^{d-i-1}(R/\qin)} \xra{\delta} 0^{*_\rho}_{H_\fm^{d-i}(R/\qi)} \to x_{i+1}0^{*_\rho}_{H_\fm^{d-i}(R/\qi)} \to 0.$$ 
This comes from analyzing the usual short exact sequence induced by multiplication-by-$x_{i+1}$  $$0 \to R/(\qi:x_{i+1}) \to R/\qi \to R/\qin \to 0$$ using the isomorphisms $H_\fm^{d-i}(R/(\qi \colon x_{i+1})) \cong H_\fm^{d-i}(R/\qi)$, so the map $\alpha$ is induced by $R/\qi \to R/\qin$.

By definition, $\delta\left(0^\rho_{H_\fm^{d-i-1}(R/\qin)}\right) \subseteq 0^\rho_{H_\fm^{d-i}(R/\qi)}$. This gives the following diagram whose bottom row is exact.

\begin{center}
\begin{tikzcd}
0^{\rho}_{H_\fm^{d-i-1}(R/\qin)} \arrow{r}{\delta} \arrow[d,hookrightarrow] & 0^{\rho}_{H_\fm^{d-i}(R/\qi)} \arrow[d,hookrightarrow] \arrow{r}{\cdot x_{i+1}} & x_{i+1}0^{\rho}_{H_\fm^{d-i}(R/\qi)} \arrow[d,hookrightarrow]\arrow[r,"="]& 0 \\
0^{*_\rho}_{H_\fm^{d-i-1}(R/\qin)} \arrow{r}{\delta}  & 0^{*_\rho}_{H_\fm^{d-i}(R/\qi)} \arrow{r}{\cdot x_{i+1}} & x_{i+1}0^{*_\rho}_{H_\fm^{d-i}(R/\qi)} \arrow{r} & 0
\end{tikzcd}
\end{center} Note that $x_{i+1}0^{\rho}_{H_\fm^{d-i}(R/\qi)} = 0$ by hypothesis and \eqref{eq:0rho}. Any element in $0^{\rho}_{H_\fm^{d-i}(R/\qi)}$ lifts along $\delta$ to an element $\eta$ in $0^{*_\rho}_{H_\fm^{d-i-1}(R/\qin)}$ and it suffices to show that $\eta$ comes from the a priori smaller module $0^{\rho}_{H_\fm^{d-i-1}(R/\qin)}$ in order to show that $\fq^F/\fq \to 0^F_{H^d_\fm(R)}$ is surjective. 

For each $e \in \mathbb{N}$ consider the commutative diagram 
\begin{center}
        \begin{tikzcd}
            0 \arrow{r} & R/\fq_{\leq i}:_R x_{i+1} \arrow{r}{\cdot x_{i+1}} \arrow{d}{\rho^e} & R/\fq_{\leq i} \arrow{r}\arrow{d}{\rho^e} & R/\fq_{\leq i+1} \arrow{r}\arrow{d}{\rho^e} & 0 \\
            0 \arrow{r} & R/\fq_{\leq i}^{\fbp{p^e}}:_R x_{i+1}^{p^e} \arrow{r}{\cdot x_{i+1}^{p^e}} & R/\fq_{\leq i}^{\fbp{p^e}} \arrow{r} & R/\fq_{\leq i+1}^{\fbp{p^e}} \arrow{r} & 0 
        \end{tikzcd}
    \end{center} where the vertical maps are $p^e$-linear and the horizontal maps are $R$-linear. For each $i$ we have the following induced diagram of local cohomology modules.
   {\small
    \begin{center}
        \begin{tikzcd}
            0 \arrow{r} & H^{d-i-1}_\fm(R/\fq_{\leq i}) \arrow{r}{\alpha} \arrow{d}{\rho^e} & H^{d-i-1}_\fm(R/\fq_{\leq i+1}) \arrow{r}{\delta_0}\arrow{d}{\rho^e} & H^{d-i}_\fm(R/\fq_{\leq i}) \arrow{r}{\cdot x_{i+1}} \arrow{d}{\rho^e} & H^{d-i}_\fm(R/\fq_{\leq i}) \arrow{r}\arrow{d}{\rho^e}  & 0 \\
            0 \arrow{r} & H^{d-i-1}_\fm(R/\fq_{\leq i}^{\fbp{p^e}}) \arrow{r}{\alpha_e}  & H^{d-i-1}_\fm\left(R/\fq_{\leq i+1}^{\fbp{p^e}}\right) \arrow{r}{\delta_e} & H^{d-i}_\fm\left(R/\fq_{\leq i}^{\fbp{p^e}}\right) \arrow{r}{\cdot x_{i+1}^{p^e}} & H^{d-i}_\fm\left(R/\fq_{\leq i}^{\fbp{p^e}}\right) \arrow{r} &  0 
        \end{tikzcd}
    \end{center}
}
    Let $\xi \in 0^{\rho}_{H^{d-i}_\fm(R/\fq_{\leq i})}$. As $x_{i+1}\xi =0$ we have $\xi = \delta_0(\xi')$ for some $\xi' \in H^{d-i-1}_\fm(R/\fq_{\leq i+1})$. For $e \gg 0$, $\rho^e(\xi) = 0$, thus $\rho^e(\xi')$ lifts to $\xi_e'' \in H^{d-i-1}_\fm(R/\fq_{\leq i}^{\fbp{p^e}})$. It suffices to see that this element is nilpotent under the relative Frobenius. However, one can easily show more strongly that for all $i < d$ and $1 \le t \le d-1$, $H^{d-i-t}_\fm(R/\fq_{\leq i})$ is nilpotent under $\rho^e$ for $e \gg 0$ from  \cite[Theorem 4.2]{PQ19}, as $R$ is Buchsbaum and every system of parameters is filter-regular.
\end{proof}

\subsection{\texorpdfstring{$F$}{F}-Buchsbaum and tight Buchsbaum Cohen--Macaulay rings}
Finally, we consider Condition \ref{c} of \cref{thm:mainthm-restated}, i.e. a characterization in the Cohen--Macaulay setting which will allow for practical testing of these notions in the next section. Recall that a \emph{parameter test element} for tight closure is an element $c\in R^\circ$ such that $c\fq^*\subseteq \fq$ for all parameter ideals $\fq\subseteq R$. The characterization of Cohen--Macaulay tight Buchsbaum rings below relies on the existence of a parameter ideal generated by test elements. This occurs for excellent reduced local rings that have isolated non-$F$-rational loci by \cite[Theorem 3.9]{Vel95}, a condition which is satisfied by tight Buchsbaum rings but not necessarily by $F$-Buchsbaum rings (see e.g. \cref{ex:CGM}). The characterization of the latter thus works with a weaker replacement.
\begin{thm}\label{cor:gor}
    Let $(R,\fm,k)$ be an excellent Cohen--Macaulay local ring of the prime characteristic $p>0$. 
    \begin{enumerate}[label=(\alph*)]
        \item $R$ is tight Buchsbaum if and only if there exists a parameter ideal $\fq\subseteq R$ generated by parameter test elements such that $\fq:\fq^*\supseteq\fm$.\label{cor:gor-1}
        \item  The ring $R$ satisfies condition \ref{a} of \cref{thm:mainthm-restated} if and only if there exists a parameter ideal $$\mathfrak{Q}\subseteq \bigcap_{\fq} (\fq:\fq^F)=:J$$ (where the intersection is taken over all parameter ideals $\fq\subseteq R$) such that $\fm\subseteq \mathfrak{Q}:\mathfrak{Q}^F$. \label{cor:gor-2}
    \end{enumerate}

As Cohen--Macaulay rings are weakly $F$-nilpotent, \cref{cor:gor} shows all conditions of \cref{thm:mainthm} are equivalent in this setting.
\end{thm}
\begin{proof} 
For \ref{cor:gor-1}, given a parameter ideal $\fq = (x_1, \ldots, x_d)$ we have \[ \varinjlim_n \dfrac{\fq_n^*}{\fq_n} \cong 0^*_{H^d_{\fm}(R)},\] where as usual $\fq_n = (x_1^n, \ldots, x_d^n)$. Further, all maps in direct systems are injective as $R$ is Cohen--Macaulay. Hence $\tau_{\text{par}}(R) = \cap_{\fq} (\fq : \fq^*) = \mathrm{Ann}_R 0^*_{H^d_{\fm}(R)}$. When $\fq$ is a parameter test ideal we have $\tau_{par}(R) = \fq : \fq^*$. The result then follows from \Cref{thm:ma-pham}.

For \ref{cor:gor-2}, note that the only if direction is clear. For the converse, let $\mathfrak{Q} = (x_1,\dots, x_d)$ be a parameter ideal where each $x_i\in J$. For $t\geq 1$ let $\mathfrak{Q}_t = (x_1^t,\dots, x_d^t)$. Since $\{\mathfrak{Q}_t\}$ is cofinal with the set of powers of the maximal ideal $\fm$, by a standard argument we have $J = \bigcap_t (\mathfrak{Q}_t:\mathfrak{Q}_t^F)$. We now claim that $\mathfrak{Q}:\mathfrak{Q}^F = J$, and by the above it suffices to show that 
\begin{align}
    \fQ:\fQ^F = \fQ_t:\fQ_t^F\label{eq:gor-1}
\end{align}
for all $t\geq 1$. Letting $x:=x_1\cdots x_d$ we first show that 
  \begin{align}
      \fQ_t^F = \fQ_t+x^{t-1}\fQ^F.\label{eq:gor-2}
  \end{align}  
    To that end, let $a\in\fQ_t^F$ and note that by the assumption that $\fQ\subseteq J$ we have $x_i\fQ_t^F\subseteq \fQ_t$ for all $i$. Hence 
    $$a\in \fQ_t:\fQ = (x_1^t,\dots, x_d^t,x^{t-1}).$$ Now write $a=b+cx^{t-1}$ where $b\in\fQ_t$ and $c\in R$. Then $cx^{t-1}=a-b\in\fQ_t^F$ so $c^{p^e} x^{p^e(t-1)}\in \fQ_t^{\fbp{p^e}} = \fQ_{tp^e}$ for all $e\gg 0$. Then $c^{p^e}\in \fQ_{tp^e}:x^{p^e(t-1)} = \fQ^{\fbp{p^e}}$. This establishes the equality \eqref{eq:gor-2}. But then \eqref{eq:gor-1} follows because $r\fQ_t^F\subseteq\fQ_t$ if and only if $r x^{t-1} \fQ^F\subseteq \fQ_t$, which in turn holds if and only if $r\fQ^F \subseteq \fQ_t:x^{t-1} = \fQ$. We have shown that $\fQ:\fQ^F = J$ as desired.
\end{proof}

\section{Examples}\label{sec:examples} 
We conclude in this section by illustrating that many of the arrows in the left half of the diagram in \cref{diagram:implications} cannot be reversed, even in the case of an isolated singularity.

\begin{xmp}[A tight Buchsbaum ring which is not CMFI]\label{xmp: F-Buchs but not CMFI}
Let $k$ be a perfect field of characteristic $p>0$. Let $R=k\llbracket x,y,z,w\rrbracket/(xz,xw,yz,yw)$. Clearly $R$ is $F$-pure and has an isolated singularity by the Jacobian criterion, so $R$ is tight Buchsbaum by \cite[Theorem 3.4]{BMS18}.
\end{xmp}

\begin{xmp}[A Cohen--Macaulay non-$F$-injective tight Buchsbaum ring]
    Let $p\equiv 2\mod 3$ and let $R = \mathbb{F}_p\llbracket x,y,z\rrbracket/(x^3+y^3+z^3)$. Then $\tau_{\text{par}}(R) = \fm$, so $R$ is tight Buchsbaum by \cref{thm:ma-pham}. However, $R$ is not $F$-injective by Fedder's criterion \cite{Fed83}.
\end{xmp}

The following example is an instance of the weakly normal construction from \cite[Appendix]{CGM89}.
\begin{xmp}[An $F$-Buchsbaum ring which is not tight Buchsbaum]\label{ex:CGM}
        Let $k =\overline{k}$ be an algebraically closed field of characteristic $p>2$, let $K=k(u,v)$, and let $L = K[y]/(y^{2p}+y^p u -v)$. Then the ring $R$ given by the pullback diagram 
    \begin{equation*}
        \begin{tikzcd}
            K+(x_1,x_2)L\llbracket x_1,x_2\rrbracket\arrow[r]\arrow[d,twoheadrightarrow]& L\llbracket x_1,x_2\rrbracket\arrow[d,twoheadrightarrow]\\
            K\arrow[r]& L
        \end{tikzcd}
    \end{equation*}
is $F$-injective (by \cite[Example 3.5]{MSS17}) but not $F$-pure. It's not $F$-pure because it's not (WN1); that is, the normalization map $R\to L\llbracket x_1,x_2\rrbracket$ ramifies in codimension one \cite[Theorem 7.3]{SZ13}. Note that $R$ is Buchsbaum using either \cite[Corollary 1.3]{Ma15} or \cite[Theorem 3.1]{GW85}. Thus, $R$ is $F$-Buchsbaum by \Cref{rmk: f-inj  + buchs = f-buchs}. To see that $R$ is not tight Buchsbaum, note that $R$ is not regular in codimension one since $R$ is not (WN1). It follows that the non-$F$-rational locus of $R$ is strictly bigger than $\{\fm\}$.
\end{xmp}

\begin{xmp}[An $F$-Buchsbaum ring which is neither $F$-injective nor tight Buchsbaum]  \label{Hyp4}
    Let $$R = \mathbb{F}_p\llbracket x,y,z\rrbracket/(x^4+y^4+z^4)$$ with $p\equiv 1\mod 4$. The (parameter) test ideal is given by $\tau_{\text{par}}(R) = (x,y,z)^2$ (see e.g. \cite[209]{Hun98}), so $R$ is not tight Buchsbaum by \cite[Theorem 5.1]{MQ22}. By Fedder's criterion, $R$ is not $F$-injective.
    
    However, we claim that $R$ is $F$-Buchsbaum. To see this, consider the parameter ideal $\fq = (y^2,z^2)R$. Then one can show $\fq+(x^3yz) = \fq^F$ and that $\fq:\fq^F =  \fm$. Since $y^2$ and $z^2$ are test elements, \cref{cor:gor}\ref{cor:gor-2} implies that $R$ is $F$-Buchsbaum.
\end{xmp}

\begin{xmp}[A $2$-dimensional normal Gorenstein ring which is not $F$-Buchsbaum]
    Let $R = \mathbb{F}_2[ x,y,z]/(x^5+y^5+z^5)$ with homogeneous maximal ideal $\fm=(x,y,z)$. It is easily verified that the Frobenius map on $[H^2_\fm(R)]_0$ is nilpotent; since $R$ has an isolated singularity at $\fm$, we see that $R_\fm$ is $F$-nilpotent by \cite[Example 2.8(1)]{ST17}. It follows that $A:=\widehat{R_\fm} =\mathbb{F}_2\llbracket x,y,z\rrbracket/(x^5+y^5+z^5)$ is $F$-nilpotent by \cite[Proposition 2.8(4)]{PQ19}. Consider the two parameter ideals $\fq_1=(x,y)A$ and $\fq_2=(x^2,y^2)A$. Note that $\fq_1^F = \fq_1^* = (x,y,z^2)$ from \cite[Theorem A]{PQ19} and \cite[p. 209]{Hun98}. One checks, however, that $z^3,xyz^2\in\fq_2^F$ so that
\begin{align*}
    \ell_A(\fq_1^F/\fq_1) & = 3, \text{ and}\\
    \ell_A(\fq_2^F/\fq_2) & \geq \ell_A\left(\frac{(y^{2},\,x^{2},\,z^{3},\,x\,y\,z^{2})}{\fq_2}\right) = 9
\end{align*}
from which we conclude that $A$ is not $F$-Buchsbaum.
\end{xmp}

\begin{xmp}[A reduced Buchsbaum ring which is not $F$-Buchsbaum]
`Let $(R,\fm)$ be a reduced Cohen--Macaulay local ring which is not $F$-injective and whose non-$F$-injective locus is given by $V(\fa)$. Then the non-$F$-injective locus of $R\llbracket x\rrbracket$ is defined by the non-maximal ideal $\fa R\llbracket x\rrbracket$, hence $R\llbracket x\rrbracket$ is not $F$-Buchsbaum.
\end{xmp}

\printbibliography
\end{document}